\documentclass[a4paper,reqno]{sn-jnl}

\usepackage{graphicx}%
\usepackage{multirow}%
\usepackage{amsmath,amssymb,amsfonts}%
\usepackage{amsthm}%
\usepackage{mathrsfs}%
\usepackage[title]{appendix}%
\usepackage{xcolor}%
\usepackage{mathtools}
\usepackage[utf8]{inputenc}
\usepackage{hyperref}
\usepackage[british]{babel}
\usepackage{enumitem}
\usepackage{color}
\let\voo\v

\DeclarePairedDelimiter{\norm}{\|}{\|}
\DeclarePairedDelimiter{\snorm}{|}{|}
\newcommand{\na}{\nabla}

\def\eps{\varepsilon}
\def\la{\langle}
\def\ra{\rangle}
\def\ran{\ra^n}

\def\div{\operatorname{div}}

\def\R{\mathbb{R}}

\def\Itau{\mathcal{I}_\tau}
\def\I{I}
\def\Th{\mathcal{T}_h}
\def\Vh{\mathcal{V}_h}

\def\dt{\partial_t}

\def\v{\mathbf{v}}

\def\D{\mathcal{D}}
\def\E{\mathcal{E}}

\def\softd{{\leavevmode\setbox1=\hbox{d}%
		\hbox to 1.05\wd1{d\kern-0.4ex{\char039}\hss}}}

\def\bmu{\bar{\mu}}
\def\hbmu{\hat{\bar{\mu}}}

\def\bphi{\bar{\phi}}
\def\hbphi{\hat{\bar{\phi}}}
\def\bq{\bar{q}}
\def\hbq{\hat{\bar{q}}}

\def\ddta{\;\mathrm{d}s}
\def\dx{\;\mathrm{d}x}

\newtheorem{lemma}{Lemma}
\newtheorem{problem}[lemma]{Problem}
\newtheorem{theorem}[lemma]{Theorem}
\theoremstyle{definition}

\newtheorem{assumption}[lemma]{Assumption}

\begin{document}

\title[Error analysis for a viscoelastic phase separation model]{Error analysis for a viscoelastic phase separation model}


\author*[1]{\fnm{Aaron} \sur{Brunk}}\email{abrunk@uni-mainz.de}

\author[2,3]{\fnm{Herbert} \sur{Egger}}\email{herbert.egger@jku.at}
\equalcont{These authors contributed equally to this work.}

\author[3]{\fnm{Oliver} \sur{Habrich}}\email{oliver.habrich@jku.at}
\equalcont{These authors contributed equally to this work.}

\author[1]{\fnm{Maria} \sur{Luk\'a\voo{c}ov\'a-Medvi\softd ov\'a}}\email{lukacova@uni-mainz.de}
\equalcont{These authors contributed equally to this work.}

\affil*[1]{\orgdiv{Institute of Mathematics}, \orgname{Johannes Gutenberg-University}, \orgaddress{\city{Mainz}, \country{Germany}}}

\affil[2]{\orgdiv{}, \orgname{Johann Radon Institute for Computational and Applied Mathematics}, \orgaddress{\city{Linz}, \country{Austria}}}

\affil[3]{\orgdiv{Institute for Numerical Mathematics}, \orgname{Johannes Kepler University}, \orgaddress{\city{Linz}, \country{Austria}}}


\abstract{We consider systematic numerical approximation of a viscoelastic phase separation model that describes the demixing of a polymer solvent mixture. An unconditionally stable discretisation method is proposed based on a finite element approximation in space and a variational time discretization strategy. The proposed method preserves the energy-dissipation structure of the underlying system exactly and allows to establish a fully discrete nonlinear stability estimate in natural norms based on the concept of relative energy. These estimates are used to derive order optimal error estimates for the method under minimal smoothness assumptions on the problem data, despite the presence of various strong nonlinearities in the equations. The theoretical results and main properties of the method are illustrated by numerical simulations which also demonstrate the capability to reproduce the relevant physical effects observed in experiments.
}

\keywords{phase-field models, viscoelastic phase separation, variational time discretization, finite elements, relative energy estimates, numerical analysis}


\pacs[MSC Classification]{35K52, 35K55, 65M12, 65M60, 82C26}

\maketitle

The separation of polymer-solvent mixtures after a deep quench is strongly influenced by dynamic asymmetry, i.e., by different length and time scales of two components~\cite{Schmid2023,Tanaka2017}. This leads to meta-stable states and the formation of intermediate networks which are usually not observed in the demixing of binary fluids. 
In~\cite{Tanaka}, Tanaka made a first attempt to model such \emph{viscoelastic phase separation} phenomena by introducing an internal relaxation variable. Tanaka's model provided good qualitative agreement with experimental observations. 
In order to guarantee thermodynamic consistency, the model was further modified by Zhou et al.~\cite{ZZE}. 
The resulting system of partial differential equations is of parabolic-hyperbolic type and combines the Cahn-Hilliard equation for the volume fraction, a relaxation equation for an internal variable, the Navier-Stokes equations for the fluid flow, and the Oldroyd-B model for the viscoelastic stress. 
Well-posedness of the viscoelastic phase separation model and suitable modifications were thoroughly analyzed in~\cite{Brunk.Ex3D,Brunk.Exreg}.

\subsection*{Basic model under consideration.}
The authors of \cite{ZZE} also consider a simplified model for viscoelastic phase separation, which is a particular case of the system
\begin{alignat}{5} 
\dt \phi &= \div\big(b(\phi) \nabla \mu + c(\phi)\nabla(A(\phi)q)\big), \label{eq:chq1}\\
\mu &= -\gamma \Delta \phi + f'(\phi), \label{eq:chq2} \\
\dt q &= -\kappa(\phi) \, q - A(\phi)\div\big(c(\phi)\nabla\mu - d(\phi) \nabla (A(\phi)q ) \big) + \varepsilon \Delta q.
\label{eq:chq3} 
\end{alignat}
Here $\phi$ is the volume fraction of the polymer and solvent components, $\mu$ is the chemical potential, and $q$ denotes a pressure-like quantity related to the bulk stress in the mixture; $\gamma>0$ is the interface parameter, $f(\cdot)$ is the derivative of an internal energy of the mixture, $b(\cdot)$, $c(\cdot)$, $d(\cdot)$ are phase dependent mobility parameters, $\kappa(\cdot)$ is an inverse relaxation time, $A(\cdot)$ a bulk modulus, and $\varepsilon$ a regularization parameter. The latter is required to be positive for the analysis of the system which can also be motivated from the multi-scale modelling perspective~\cite{Sueli07}.  
For $\eps=0$, $b(\phi)=M(\phi) \phi^2 (1-\phi)^2$, $c(\phi)=M(\phi) \phi (1-\phi)$, and $d(\phi) = M(\phi)$ the above system amounts to the problem studied in~\cite[Section~III]{ZZE}. 
The first simulation results for this simplified model of viscoelastic phase separation were obtained in \cite{ZZE} by explicit time-stepping methods. A good qualitative agreement was observed with the experimental data from~\cite{Tanaka}.
In \cite{Paul,Strasser.2019}, linear-implicit energy-stable approximations were developed for the above system guaranteeing energy dissipation on the discrete level.
Extensive numerical tests were performed for the simplified model \eqref{eq:chq1}--\eqref{eq:chq3} above as well as for the full model coupled to viscoelastic fluid flow. A rigorous error analysis for corresponding discretization methods seems to be missing up to date.

\subsection*{Related work}
The problem under consideration shares similarities with cross-diffusion systems; see~\cite{Jngel2016} for an introduction and references. 
Discretization schemes preserving the underlying entropy-dissipation structure have been proposed and investigated in \cite{Braukhoff2021,Juengel2021,JngelVetter+2023,Juengel2022}. Typically, finite-volume schemes are used in this context, since they accommodate a discrete version of the chain rule which is important to deal with the nonlinearities of the model. 
Various discretisation methods have been devised for the Cahn-Hilliard equation and similar gradient systems, see e.g. \cite{Elliott1992,Feng2004}. Recent developments, like the \emph{scalar auxiliary variable} (SAV) approach \cite{SHEN2018407} or the \emph{energy quadratisation} (EQ) approach \cite{GONG2019} aim at improving computational efficiency. This is achieved by relaxing the nonlinear relation between energy and chemical potential to some extent by introducing auxiliary variables. Similar ideas are also used in \cite{LI23,Akrivis19,CHEN2019} and in \cite{Yang2020,CHEN2022,Zhang2022}.
In our previous work~\cite{brunk2021ch}, we proposed a structure-preserving discretization scheme for the Cahn-Hilliard equation and we provided a full stability and error analysis based on the relative energy estimates.

\subsection*{Scope and main contributions.}
In this paper, we propose a fully discrete approximation scheme for the system \eqref{eq:chq1}--\eqref{eq:chq3} based on a finite element approximation in space and a variational time discretization strategy. The resulting method is unconditionally energy stable, preserves the underlying energy-dissipation structure of the problem, and yields order optimal error estimates under minimal smoothness assumptions on the problem data and on the continuous solution.
The main ingredient of our analysis are \emph{relative energy estimates}, which allow us to establish discrete stability estimates in appropriate norms. Together with the variational structure of the approximation scheme, this allows to bound the discretization errors by the corresponding projection and interpolation errors, leading to the optimal quantitative a-priori error bounds. The discrete stability estimates further allow us to prove uniqueness of the discrete solution under a mild restriction on the spatial and temporal mesh size. 
A key difficulty in the analysis of the problem lies in the various nonlinear terms of the model which are, however, important to obtain good qualitative agreement with experimental data~\cite{Tanaka,ZZE}.
This is resolved here by the nonlinear discrete stability analysis mentioned above. The technicalities therefore are shifted to estimation of certain residual terms, which can be analysed independently.

\subsection*{Outline.}
The remainder of the manuscript is organized as follows: 
In Section~\ref{sec:prelim}, we introduce our notation and basic assumptions. We also collect some results about the analysis of the problem~\eqref{eq:chq1}--\eqref{eq:chq3}.
In Section~\ref{sec:main}, we then present our discretization scheme and state the main results of the paper. 
The essential parts of the proofs are elaborated in Sections~\ref{sec:existence}--\ref{sec:uniqueness}, and further technical details are provided in the appendix. 
For illustration of our theoretical findings, we present some numerical tests in Section~\ref{sec:num} and then close with a brief discussion.

\section{Notation, assumptions, and preliminaries} 
\label{sec:prelim}

We consider a periodic setting and assume that $\Omega \subset \R^d$ is a hypercube in dimension $d=2,3$ which is identified with the $d$-dimensional torus. 
Functions on $\Omega$ are assumed periodically extendable to $\R^d$ under preservation of class.  
We denote by $L^p(\Omega)$, $W^{k,p}(\Omega)$ the corresponding Lebesgue and Sobolev spaces and write $\norm{\cdot}_{0,p}$, $\norm{\cdot}_{k,p}$ for the respective norms. 
If the meaning is clear from the context, we will sometimes omit the symbol $\Omega$ and write $L^p$ for $L^p(\Omega)$. Analogous notation holds for other spaces.
We further abbreviate $H^k(\Omega)=W^{k,2}(\Omega)$ and $\norm{\cdot}_k = \norm{\cdot}_{k,2}$. 
The corresponding dual spaces are denoted by $H^{-s}_p(\Omega)=H^s_p(\Omega)'$. 
Note that for $s=0$, we have $H^s(\Omega) = H^{-s}(\Omega)=L^2(\Omega)$, where we tacitly identified $L^2(\Omega)$ with its dual space. 
The norm of the dual spaces are given by 
\begin{align} \label{eq:dualnorm}
    \norm{r}_{-s} = \sup_{v \in H^s(\Omega)} \frac{\la r, v\ra}{\|v\|_{s}},
\end{align}
where $\langle \cdot, \cdot\rangle$ denotes for the duality product on $H^{-s}(\Omega) \times H^s(\Omega)$ for $s \ge 0$.  The same symbol will be used for the scalar product of $L^2(\Omega)$. 
For functions $u,v \in H^0(\Omega) = L^2(\Omega)$, we use the same symbol $\la u,v \ra = \int_\Omega u v \, dx$ to denote the scalar product on $L^2(\Omega)$.
We further denote by $L^p(a,b;X)$, $W^{k,p}(a,b;X)$,  and
$H^k(a,b;X)$ the Bochner spaces of integrable or differentiable functions on the time interval $(a,b)$ with values in some Banach space $X$. If $(a,b)=(0,T)$, we will omit reference to the time interval and briefly write $L^p(X)$, for instance.

The following assumptions on the problem data will be used throughout the manuscript. 
\begin{assumption} \phantom{e}\\[-1em]
 \begin{itemize}[leftmargin=1cm] 
     \item[(A0)] $\Omega \subset \R^d$, $d=2,3$ is a hypercube and identified with the $d$-dimensional torus. Functions on $\Omega$ are assumed periodically extendable to $\R^d$; $T>0$ given.
     \item[(A1)]  $b \in C^1(\R)$ with $0 < b_1\leq b(s) \leq b_2$ for all $s\in\R$ and $\norm{b'}_{0,\infty}\leq b_3$; $c \in C^{1}(\R)$ with $0 \le c_1\leq c(s) \leq c_2$ for all $s\in\R$ and $\norm{c'}_{0,\infty}\leq c_3$;  $d(s)=d_0>0$ and $b(s) \ge c(s)^2/d_0 + \eps$. 
     \item[(A2)] $f \in C^4(\R)$ with $f(s),f''(s)\geq -f_1$ and $|f^{(k)}(s)|\le f_2^{(k)} + f_3^{(k)} |s|^{4-k}$ with $f_1,f_{i}^{(k)}\geq 0$.
     \item[(A3)] $A\in C^2(\R)$ with $0 \leq A_1 \leq A(s) \leq A_2$ and $\norm{A^{(k)}}_{0,\infty}\leq A_{k+2}$ for $k=1,2$.
     \item[(A4)] $\kappa\in C^1(\R)$ with $0 < \kappa_{1}\leq \kappa(s) \leq \kappa_{2}$ and $\norm{\kappa'}_{0,\infty}\leq \kappa_{3}$. 
     \item[(A5)] $\gamma,\varepsilon>0$ constant.
 \end{itemize}   
\end{assumption}
The above assumptions allow to prove the existence of weak solutions to \eqref{eq:chq1}--\eqref{eq:chq3} for appropriate initial values. Corresponding results for a more complex model have been obtained in \cite{Brunk.Ex3D,Brunk.Exreg}. 

\subsection*{Variational characterization}
As a starting point for designing a suitable discretization scheme, we note that sufficiently regular periodic solutions of \eqref{eq:chq1}--\eqref{eq:chq3} satisfy 
\begin{align}
\la \dt\phi,\psi\ra + \la b(\phi)\na\mu - c(\phi) \na(A(\phi)q),\na\psi\ra &= 0, \label{eq:weak1}\\
\la \mu,\xi\ra - \gamma\la \na\phi,\na\xi\ra - \la f'(\phi),\xi\ra &= 0,\label{eq:weak2}\\
\la \dt q,\zeta\ra + \la \kappa(\phi)q,\zeta\ra + \la d_0 \na(A(\phi)q) - c(\phi)\na\mu,\na(A(\phi)\zeta)\ra + \varepsilon \la \na q,\na\zeta\ra &= 0 \label{eq:weak3}
\end{align}
for all smooth test functions $\psi,\xi,\zeta$ and all $t\in[0,T]$. 
Note that the solution components here depend on time, while the test functions do not. 
The above identities follow immediately by testing the equations appropriately and using integration-by-parts for some of the terms. 

\subsection*{Basic properties}
The variational identities \eqref{eq:weak1}--\eqref{eq:weak3} immediately imply the following properties of sufficiently smooth solutions:
By testing with $\psi=1$, $\xi=0$ and $\zeta=0$, we get 
\begin{align} \label{eq:mass}
\frac{d}{dt} \int_\Omega \phi(t) \dx = 0,
\end{align}
which encodes the conservation of mass. 
Testing with $\psi=\mu(t)$, $\xi=\dt\phi(t)$ and $\zeta=q(t)$ on the other hand, leads to the energy dissipation identity
\begin{align} \label{eq:energy}
\frac{d}{dt} \E(\phi(t),q(t)) = - \D_{\phi(t)}(\mu(t),q(t)).
\end{align}
Here $\E(\phi,q) = \int_\Omega \frac{\gamma}{2} |\nabla \phi|^2 + f(\phi) + \frac{1}{2} |q|^2$ denotes the free energy associated to the system and $D_\phi(\mu,q) = \int_\Omega \frac{1}{d_0} |c(\phi) \nabla \mu - d_0 \nabla (A(\phi)|^2 + (b(\phi) - c(\phi)^2/d_0) |\nabla \mu|^2 + \eps |\nabla q|^2 + \kappa(\phi) |q|^2 \dx$ the corresponding dissipation functional.
Both properties are important for proving the existence of weak solutions on the continuous level. 
They are a direct consequence of the variational characterization~\eqref{eq:weak1}--\eqref{eq:weak3} of solutions and can be preserved by appropriate discretization schemes.

\section{Proposed method and main results}\label{sec:main}

We start by introducing additional notation, assumptions and the approximation method for our model problem. Afterwards, we state our main results and briefly comment on the main arguments of the proofs, which are detailed in the following sections.

\subsection*{Notation and assumptions.}
For the space discretization, we assume that
\begin{itemize}\itemsep0ex  \setlength{\itemindent}{0.3em}
    \item[(A6)] $\Th$ is a geometrically conforming and quasi-uniform partition of $\Omega$ into simplices that can be extended periodically to periodic extensions of $\Omega$. 
\end{itemize}
By quasi-uniform, we mean that there exists a constant $\sigma>0$ such that $\sigma h \le \rho_K \le h_K \le h$ for all $K \in \Th$, where $\rho_K$ and $h_K$ are the inner-circle radius and diameter of the element $K \in \Th$ and $h=\max_{K \in \Th} h_T$ is the global mesh size \cite{BrennerScott}. 
We then denote by 
\begin{align*}
    \Vh &:= \{v \in H^1(\Omega) : v|_K \in P_2(K) \quad \forall K \in \Th\},
\end{align*}
the space of continuous periodic piecewise quadratic polynomials on $\Th$. 
By $\pi_h^0 : L^2(\Omega) \to \Vh$ and $\pi_h^1 : H^1(\Omega) \to \Vh$, we denote the $L^2$- and $H^1$-orthogonal projection operators, defined by 
\begin{alignat}{2}
\la \pi_h^0 u - u, v_h\ra &= 0 \qquad && \forall v_h \in \Vh, \label{eq:defl2proj}\\
\la \pi_h^1 u - u, v_h\ra + \la \nabla (\pi_h^1 u - u ), \nabla v_h \ra &= 0 \qquad && \forall v_h \in \Vh.\label{eq:defh1proj}
\end{alignat}
Some basic properties of these operators are again summarized in Appendix~\ref{app:proj}.
We will frequently make use of the discrete dual norm given by 
\begin{align}
\|r\|_{-1,h} := \sup_{v_h \in \Vh} \frac{\la r, v_h\ra}{\|v_h\|_1}
\end{align}
which is the discrete version of the dual norm.
For the approximation in time, we also use piecewise polynomial functions, defined on the grid
\begin{itemize}\itemsep0ex  \setlength{\itemindent}{0.3em}
\item[(A7)] $\Itau:=\{0=t^0,t^1,\ldots,t^N=T\}$ \quad with time steps $t^n = n \tau$ and step size $\tau = T/N$.
\end{itemize}
We note that 
non-uniform time steps could be considered with minor modifications of the arguments presented in the following. 
We write $I_n:=(t^{n-1},t^n)$ for the $n$-th time interval and use $\la a, b \ran = \int_{I^n} \la a, b \ra \ddta$ to abbreviate the integral over $I^n$. 
We further introduce the spaces
\begin{align}
P_k(\Itau;X) 
\qquad \text{and} \qquad 
P_k^c(\Itau;X) = P_k(\Itau;X) \cap C(0,T;X),
\end{align}
consisting of all discontinuous, respectively, continuous piecewise polynomial functions of degree less or equal then $k$ on the time grid $\Itau$, with values in some vector space $X$.
We write
$\I_\tau^1:H^1(0,T;X)\to P_1^c(\Itau;X)$ and 
$\bar \pi_\tau^0 : L^2(0,T;X) \to P_0(\Itau;X)$ for the piecewise linear interpolation, respectively, the piecewise constant projection of functions in time. 
Some important properties of these operators are again summarized in Appendix~\ref{app:proj}. 
Throughout the presentation, the bar symbol $\bar u$ is used to indicate functions in $P_0(\Itau;X)$ which are piecewise constant in time.
For ease of presentation, we use the same symbol 
$\bar u = \bar \pi_\tau^0 u$ also to abbreviate the piecewise constant projection in time of a function $u \in L^2(0,T;X)$. 

\subsection*{Discretization method}
As an approximation of the initial value problem for
\eqref{eq:chq1}--\eqref{eq:chq3}, we consider the following scheme, which is motivated by the variational characterization of solutions.

\begin{problem}\label{prob:full}
Let (A0)--(A7) hold and $\phi_{h,0},q_{h,0} \in \Vh$ be given. 
Find $\phi_{h,\tau},q_{h,\tau} \in P_1^c(\Itau;\Vh)$ and $\bar\mu_{h,\tau} \in P_0(\Itau;\Vh)$ such that $\phi_{h,\tau}(0)=\phi_{h,0}$ and $q_{h,\tau}(0)=q_{h,0}$, and such that 
\begin{align}
\la \dt\phi_{h,\tau}, \bar\psi_{h,\tau}\ran &=  -\la b(\bar\phi_{h,\tau})\na\bmu_{h,\tau} - c(\bar\phi_{h,\tau})\na(A(\bar\phi_{h,\tau}) \bar q_{h,\tau}),\na\bar\psi_{h,\tau} \ran, \label{eq:pg1}\\
\la \bar\mu_{h,\tau},\bar\xi_{h,\tau}  \ran & = \gamma\la \nabla\bar\phi_{h,\tau},\nabla\bar\xi_{h,\tau}\ran + \la f'(\phi_{h,\tau}),\bar\xi_{h,\tau}\ran, \label{eq:pg2}\\
\la \dt q_{h,\tau}, \bar\zeta_{h,\tau}\ran & =  - \la d_0 \na(A(\bar\phi_{h,\tau}) \bar q_{h,\tau}) - c(\bar\phi_{h,\tau})\nabla\bar\mu_{h,\tau},\nabla(A(\bar\phi_{h,\tau})\bar\zeta_{h,\tau})\ran\notag \\
& \qquad \qquad - \la \kappa(\bar\phi_{h,\tau})\bar q_{h,\tau},\bar\zeta_{h,\tau}\ran - \varepsilon \la \na \bar q_{h,\tau} \na\bar\zeta_{h,\tau}\ran, \label{eq:pg3}
\end{align}
for all test functions $\bar\psi_{h,\tau}$, $\bar \xi_{h,\tau}$, $\bar\zeta_{h,\tau} \in P_0(I_n;\Vh)$ and all time steps $1 \le n \le N$.
 Let us recall that $\la a,b\ran = \int_{t^{n-1}}^{t^n} \la a,b\ra \, ds =\int_{t^{n-1}}^{t^n} \int_\Omega a \cdot b \dx \ddta$ is used for the abbreviation of space-time integrals.
 \end{problem}

In the $n$-th time step of the method, the values $\phi_{h,\tau}(t^{n-1})$, $q_{h,\tau}(t^{n-1})$ are known, and one has to find $\phi_{h,\tau}(t^n)$, $q_{h,\tau}(t^n)$ and $\bar \mu_{h,\tau}(t^n)$. The above scheme thus amounts to a fully implicit time-stepping scheme. Solvability will be discussed below.

\subsection*{Main results.}
As a first step of our analysis, let us comment on the well-posedness of the discrete problem and its preservation of the basic properties of the underlying system. 

\begin{theorem}\label{thm:ex}
Let (A0)--(A7) hold. 
Then for any $\phi_{h,0},q_{h,0} \in \Vh$, Problem~\ref{prob:full} has at least one solution. Moreover, any such solution conserves mass and dissipates energy, i.e.,
\begin{align}
 &\la \phi_{h,\tau}(t^n),1\ra = \la \phi_{h,\tau}(t^m),1\ra 
  \text{ and } 
 \E(\phi_{h,\tau},q_{h,\tau})\big|_{t^m}^{t^n} = -\int_{t^m}^{t^n} \D_{\bar \phi_{h,\tau}}(\bar \mu_{h,\tau},\bar q_{h,\tau}) \, \ddta \label{eq:discstructure}
\end{align}   
for all $0 \le m \le n \le N$. Here the energy and dissipation functionals $\E(\phi,q)$, $\D_{\phi}(\mu,q)$ are defined after \eqref{eq:energy}. 
Furthermore, solutions can be uniformly bounded by 
\begin{align} \label{eq:full_apriori}
\|\phi_{h,\tau}\|_{L^\infty(H^1)}^2 &+ \|q_{h,\tau}\|_{L^\infty(L^2)}^2  + \|\bq_{h,\tau}\|_{L^2(H^1)}^2 + \|\bmu_{h,\tau}\|_{L^2(H^1)}^2 
\le C'. 
\end{align}
The constant $C'=C'(\|\phi_{h,0}\|_{H^1},\|q_{h,0}\|_{L^2})$ depends only on the bounds of the coefficients appearing in the assumptions and the norm of the initial data.
\end{theorem}
The two identities \eqref{eq:discstructure} follow immediately from the variational characterization of discrete solutions and insertion of appropriate test functions; compare with the continuous level. The energy-dissipation identity provides a-priori bounds on the solution, which allows to establish existence by a fixed-point argument. The complete proof will be presented in Section~\ref{sec:existence}.
We continue by stating the main result on the error estimates.
\begin{theorem}\label{thm:full_error_est}
Let (A0)--(A7) hold and let $(\phi,\mu,q)$ be a smooth solution of \eqref{eq:weak1}-\eqref{eq:weak3} satisfying 
\begin{align}
    \phi&\in H^{2}(0,T;H^1(\Omega))\cap H^1(0,T;H^3(\Omega)), \label{eq:reg1}\\ 
    \mu& \in H^2(0,T;H^1(\Omega))\cap  L^\infty(0,T;W^{1,3}(\Omega))\cap L^2(0,T;H^3(\Omega)), \label{eq:reg2}\\
    q&\in H^2(0,T;H^1(\Omega))\cap  L^\infty(0,T;W^{1,3}(\Omega))\cap L^2(0,T;H^3(\Omega)). \label{eq:reg3} 
\end{align}
Furthermore, let $(\phi_{h,\tau},\bar \mu_{h,\tau},q_{h,\tau})$ be a solution of Problem~\ref{prob:full} for some $h,\tau>0$ and with initial values given by $\phi_{h,0}=\pi_h^1 \phi(0)$ and $q_{h,0}=\pi_h^0 q(0)$. Then 
\begin{align*}
\norm{\phi_{h,\tau}-\phi}_{L^\infty(H^1)}^2 +  \norm{q_{h,\tau}-q}_{L^\infty(L^2)}^2 &+ \norm{\bmu_{h,\tau}-\bar\mu}_{L^2(H^1)}^2 \\
&+ \norm{\bq_{h,\tau}-\bar q}_{L^2(H^1)}^2    \leq C''(h^4 + \tau^4).
\end{align*}
The constant $C''$ is independent of $h$ and $\tau$. 
Moreover, for the choice $h=c'\tau$ with $c'>0$ independent of $h$ and $\tau$, the discrete solution is unique. 
\end{theorem}

We emphasize that in contrast to the existence of discrete solutions established in Theorem~\ref{thm:ex}, we need to assume the existence of a sufficiently regular exact solution to obtain the uniqueness result.
The detailed proof is given in Sections~\ref{sec:stability}--\ref{sec:uniqueness}. For a better orientation, we point out the main steps already here: Following standard practice, we decompose the error into a projection error and a discrete evolution error. The first can be treated by standard arguments which also reveal that the regularity assumptions of the theorem are rather sharp.  
In Section~\ref{sec:stability}, we establish a nonlinear stability estimate for the discrete problem which allows us to bound the discrete evolution error by certain residuals which arise when replacing the discrete solution in Problem~\ref{prob:full} by projections of the continuous solution.
The residuals are identified and corresponding bounds are stated in Section~\ref{sec:error}, which allows to conclude the global error estimates. 
Uniqueness of the discrete solution is proven in Section~\ref{sec:uniqueness} using similar arguments. 

\section{Proof of Theorem~\ref{thm:ex} } \label{sec:existence}

\subsection*{Basic properties of discrete solutions.}
We start with establishing the two important identities \eqref{eq:discstructure} stated in the theorem. 
It suffices to consider a single time step, e.g.,  the case $m=n-1$, $1 \le n \le N$.
Let $\phi_{h,\tau}$, $q_{h,\tau} \in P_1(I^n;\Vh)$ and $\bar \mu_{h,\tau} \in P_0(I^n;\Vh)$ solve \eqref{eq:pg1}--\eqref{eq:pg3}. 
By choosing $\bar \psi_{h,\tau}=1$, $\bar \psi_{h,\tau}=0$, and $\bar \zeta_{h,\tau} = 0$ as test functions in the discrete variational identities, we obtain 
\begin{align*}
\la \phi_{h,\tau},1\ra\big|_{t^{n-1}}^{t^n} 
&= \la \dt \phi_{h,\tau},1\ran 
= -\la b(\bar \phi_{h,\tau}) \nabla \bar \mu_{h,\tau} - c(\bar \phi_{h,\tau}) \nabla (A(\bar \phi_{h,\tau}) \bar q_{h,\tau}), \nabla 1 \ran = 0.
\end{align*}
In the first step, we used the fundamental theorem of calculus and the notation $\la a,b\ran = \int_{t^{n-1}}^{t^n} \la a,b\ra \ddta$. 
This already yields conservation of mass for a single time interval. The general case follows by induction over $n$ and using the continuity of $\phi_{h,\tau}$ in time.
In a similar manner, we obtain 
\begin{align*}
\E(\phi_{h,\tau},q_{h,\tau}) \big|_{t^{n-1}}^{t^n} 
&= \gamma \la \nabla \bar \phi_{h,\tau}, \nabla \dt \phi_{h,\tau}\ran + \la f'(\phi_{h,\tau}), \dt \phi_{h,\tau}\ran + \la \dt q_{h,\tau}, \bar q_{h,\tau}\ran. 
\end{align*}
For the first and last term, we used that $\dt\nabla\phi_{h,\tau}, \dt q_{h,\tau} \in P_0(I^n;\Vh)$ are constant in time on the interval $I^n$ and hence $\la \dt q_{h,\tau},q_{h,\tau}\ran = \la \dt q_{h,\tau}, \bar q_{h,\tau}\ran$ and similar for the first term, where $\bar q_{h,\tau} = \bar \pi_\tau^0 q_{h,\tau}$ is the piecewise constant projection in time. 
Using \eqref{eq:pg2} with $\bar \xi_{h,\tau}=\dt \phi_{h,\tau}$ and \eqref{eq:pg3} with $\bar \zeta_{h,\tau} = \bar q_{h,\tau}$, we obtain 
\begin{align*}
\E(\phi_{h,\tau},q_{h,\tau}) \big|_{t^{n-1}}^{t^n} 
&= \la \bar \mu_{h,\tau}, \dt \phi_{h,\tau}\ran - \la \kappa(\bar \phi_{h,\tau}) \bar q_{h,\tau}, \bar q_{h,\tau} \ran - \eps \la \nabla \bar q_{h,\tau}, \nabla \bar q_{h,\tau}\ran \\
 &
 - \la d_0 \nabla (A(\bar \phi_{h,\tau}) \bar q_{h,\tau}) - c(\bar \phi_{h,\tau}) \nabla \bar \mu_{h,\tau}, \nabla(A(\bar \phi_{h,\tau}) \bar q_{h,\tau}) \ran.
\end{align*}
Using~\eqref{eq:pg1} with $\bar \psi_{h,\tau} = \bar \mu_{h,\tau}$, the first term on the right-hand side can be replaced by $-\la b(\bar\phi_{h,\tau})\na\bmu_{h,\tau} - c(\bar\phi_{h,\tau})\na(A(\bar\phi_{h,\tau}) \bar q_{h,\tau}),\na\bar\mu_{h,\tau} \ran$. In summary, we thus obtain 
\begin{align*}
\E(&\phi_{h,\tau},q_{h,\tau})  \big|_{t^{n-1}}^{t^n} 
= -\la b(\bar \phi_{h,\tau}) \nabla \bar \mu_{h,\tau}, \nabla \bar \mu_{h,\tau} \ran + 2 \la c(\bar \phi_{h,\tau}) \nabla (A(\bar\phi_{h,\tau}) \bar q_{h,\tau}), \nabla \bar \mu_{h,\tau}\ran \\
& - \la d_0 \nabla (A(\bar\phi_{h,\tau}) \bar q_{h,\tau}), \nabla (A(\bar\phi_{h,\tau}) \bar q_{h,\tau})\ran
- \la \kappa(\bar \phi_{h,\tau}) \bar q_{h,\tau}, \bar q_{h,\tau} \ran - \eps \la \nabla \bar q_{h,\tau}, \nabla \bar q_{h,\tau}\ran.
\end{align*}
A careful inspection of the individual terms reveals that the right-hand side of this identity exactly amounts to the dissipation term $\int_{t^{n-1}}^{t^n} \D_{\bar \phi_{h,\tau}}(\bar \mu_{h,\tau}, \bar q_{h,\tau}) \ddta$. This yields the discrete energy-dissipation identity for a single time step. The general case then follows by induction.  

\subsection*{A-priori bounds}
Using assumptions (A0)--(A5), one can immediately see that 
\begin{alignat*}{2}
\|\nabla \phi\|_{L^2}^2 + \|q\|_{L^2}^2 \le C_1 \E(\phi,q) + C_2 f_1  
\qquad \text{for all } \phi,q \in H^1(\Omega),
\end{alignat*}
where $f_1$ is the lower bound for $f$ from (A2).
Furthermore, $\|\phi\|_{H^1}^2 \le C_3 \|\nabla \phi\|_{L^2}^2 + C_4 |\int_\Omega \phi \dx|^2$ by the Poincar\'e inequality. From the discrete mass conservation and energy-dissipation property, and using the positivity of the dissipation functional, we thus already obtain 
\begin{align*}
\|\phi\|_{L^\infty(H^1)}^2 + \|q\|_{L^\infty(L^2)}^2 \le C_1',
\end{align*}
with $C_1'$ depending only on the bounds of the coefficients in (A1)--(A7) and the energy and mass of the initial data. 
From the energy-dissipation identity and the bounds of the coefficients, we further get 
\begin{align*}
\|\nabla \bar \mu_{h,\tau}\|_{L^2(L^2)}^2 + \|\bar q_{h,\tau}\|_{L^2(H^1)}^2 
+ \|c(\bar\phi_{h,\tau})\na\bmu_{h,\tau}- d_0 \na(A(\bar\phi_{h,\tau})\bq_{h,\tau})\|_{L^2(L^2)}^2  \le C_2'.
\end{align*}
Here $C_2'$ again only depends on the bounds of the coefficients and the initial data.
By testing the identity~\eqref{eq:pg2} with $\bar \xi_{h,\tau}=1$, we further see that
\begin{align*}
\la \bar \mu_{h,\tau},1\ran 
&= \la f'(\phi_{h,\tau}), 1 \ran 
\le C' \tau \left(f_2^{1} + f_3^{1} \|\phi_{h,\tau}\|_{L^\infty(L^3)}^3\right).
\end{align*}
Summation over $n$, the continuous embedding of $H^1(\Omega)$ in $L^3(\Omega)$, the uniform bounds for $\|\phi_{h,\tau}\|_{L^\infty(H^1)}$ and $\|\nabla \bar \mu_{h,\tau}\|_{L^2(L^2))}$, and the Poincar\'e inequality then lead to
\begin{align*}
\|\bar \mu_{h,\tau}\|_{L^2(L^2)}^2 
&\le C_1'' \|\nabla \bar \mu_{h,\tau}\|^2_{L^2(L^2)} +  C_2'' \sum\nolimits_n \la \bar \mu_{h,\tau},1\ran 
\le C_3'
\end{align*}
with $C_3'$ again only depending on the bounds of the coefficients and the initial data. This completes the proof of a-priori bounds stated in the theorem.

\subsection*{Existence of discrete solutions}

For ease of notation, we omit the subscripts $h,\tau$ in the following. We consider the $n$-th time step and assume that $\phi^{n-1}:=\phi(t^{n-1})$ and $q^{n-1}:=q(t^{n-1})$ are already known. 
After choosing a basis for $\Vh^n$, we may rewrite \eqref{eq:pg1}--\eqref{eq:pg3} as a nonlinear system of equations $F(x)=0$ in $\R^{3N}$, and such that $\la F(x),x\ra$ amounts to testing the corresponding variational identities with $\bar \mu^{n-1/2}$, $\dt \phi^{n-1/2}$, and $q^{n-1/2}$; compare with the procedure used in the derivation of the energy-dissipation identity. 
As a consequence of the latter, we thus obtain 
\begin{align*}
\la F(x),x\ra = \E(\phi^{n}, q^n) - \E(\phi^{n-1},q^{n-1}) + \tau \D_{\bar \phi^{n-1/2}}(\bar \mu^{n-1/2},\bar q^{n-1/2}).
\end{align*}
For ease of notation, we have introduced $\phi^{n-1+\theta} := \phi^{n-1} + \theta \tau \dt \phi^{n-1/2}$ and $q^n := 2 \bar q^{n-1/2} - q^{n-1}$.
From the arguments used to derive a-priori bounds, we get that $\la F(x),x\ra \to \infty$ for $|x| \to \infty$. 
The existence of a solution then follows from \cite[Proposition~2.8]{Zeidler1}, which is a corollary to Brouwer's fixed-point theorem.

\section{A discrete stability estimate}\label{sec:stability}

In this section, we prove a nonlinear stability estimate for the discrete problem, which is one of the key ingredients for the proof of Theorem~\ref{thm:full_error_est}. 
Let $(\phi_{h,\tau},q_{h,\tau},\bar \mu_{h,\tau})$ be a solution of Problem~\ref{prob:full}, and further let 
$\hat\phi_{h,\tau}$, $\hat q_{h,\tau} \in P^c_1(\Itau;\Vh)$, and $\hbmu_{h,\tau} \in P_0(\Itau;\Vh)$ be some given functions in the corresponding spaces. 
By inserting these functions into \eqref{eq:pg1}--\eqref{eq:pg3}, we obtain 
\begin{align}
\la \dt\hat\phi_{h,\tau}, \bar\psi_{h,\tau}\ran   
&+ \la b(\bar\phi_{h,\tau})\na\hbmu_{h,\tau} - c(\bar\phi_{h,\tau}) \na(A(\bar\phi_{h,\tau})\hbq_{h,\tau}), \na\bar\psi_{h,\tau} \ran 
=\la \bar r_{1,h,\tau},\bar\psi_{h,\tau}\ran, \label{eq:pgp1}\\
\la \hbmu_{h,\tau},\bar\xi_{h,\tau} \ran &-\gamma\la \nabla\bar \phi_{h,\tau},\nabla\bar\xi_{h,\tau}\ran - \la f'(\phi_{h,\tau}),\bar\xi_{h,\tau}\ran = \la \bar r_{2,h,\tau},\bar\xi_{h,\tau}\ran, \label{eq:pgp2}\\
\la \dt\hat q_{h,\tau}, \bar\zeta_{h,\tau}\ran &+ \la d_0 \na(A(\bar\phi_{h,\tau})\hbq_{h,\tau}) -c(\bar\phi_{h,\tau})\nabla\hbmu_{h,\tau},\nabla(A(\bar\phi_{h,\tau})\bar\zeta_{h,\tau})\ran\notag \\
&+ \la \kappa(\bar\phi_{h,\tau})\hbq_{h,\tau},\bar\zeta_{h,\tau}\ran + \varepsilon \la \na\hbq_{h,\tau}, \na\bar\zeta_{h,\tau}\ran  
= \la \bar r_{3,h,\tau},\bar\zeta_{h,\tau}\ran \label{eq:pgp3}
\end{align}
for all test function $\bar\psi_{h,\tau}$, $\bar \xi_{h,\tau}$, $\bar\zeta_{h,\tau} \in P_0(I_n;\Vh)$, time steps $1 \le n \le N$, and with appropriate residuals $\bar r_{i,h,\tau}\in P_0(\Itau;\Vh)$, $i=1,2,3$, which are actually defined through these equations.

\subsection*{Relative energy}
The goal of this section is to estimate the distance between solutions of the discrete problem \eqref{eq:pg1}--\eqref{eq:pg3} and the perturbed problem \eqref{eq:pgp1}--\eqref{eq:pgp3} in terms of the residuals $\bar r_{i,h,\tau}$. 
To measure the distance, we use a relative energy functional defined by 
\begin{align*}
\E_\alpha(\phi,q|\hat\phi,\hat q) :=& \frac{\gamma}{2}\norm{\na\phi-\na\hat\phi}_0^2   + \la f(\phi) -f(\hat\phi) - f'(\hat\phi)(\phi-\hat\phi), 1 \ra \\
 + &\frac{\alpha}{2}\norm{\phi-\hat\phi}_0^2 + \frac{1}{2}\norm{q-\hat q}_0^2
\end{align*}
with parameter $\alpha = \max\{-f_1, 0\}+1$. This choice guarantees that the functional becomes convex. 
We further observe that the relative energy functional splits naturally into contributions $\E_{\alpha}^\phi(\phi_{h,\tau}|\hat\phi_{h,\tau}):=\E_\alpha(\phi_{h,\tau},0|\hat\phi_{h,\tau},0)$ and 
$ \E_{\alpha}^q(q_{h,\tau}|\hat q_{h,\tau}):=\E_\alpha(0,q_{h,\tau}|0,\hat q_{h,\tau})$ for the individual variables.
The following important estimates now immediately follow from the Taylor expansions. 
\begin{lemma}
Let (A0)--(A5) hold. Then 
\begin{align}
&\norm{\phi_{h,\tau} - \hat\phi_{h,\tau}}_1^2\leq C \, \E^\phi_{\alpha}(\phi_{h,\tau}|\hat\phi_{h,\tau}), 
\qquad \label{eq:lower_bound_rel}
\norm{q_{h,\tau} - \hat q_{h,\tau}}_0^2\leq C\, \E_{\alpha}^q(q_{h,\tau}|\hat q_{h,\tau}), 
 \text{ and} \\
&\norm{\na(\bmu_{h,\tau}-\hbmu_{h,\tau})}_0^2 + \norm{\na(\bq_{h,\tau}-\hbq_{h,\tau})}_0^2  \leq C \,  \D_{\bar\phi_{h,\tau}}(\bmu_{h,\tau} - \hbmu_{h,\tau}, \bar q_{h,\tau} - \hbq_{h,\tau}). \label{eq:lower_bound_rel_dis_full}
\end{align}
The constant $C$ depends only on the bounds of the parameters.
\end{lemma}

\subsection*{Discrete stability estimate}
We can now state and prove the main result of this section.

\begin{lemma} \label{lem:fullstab}
Let (A0)--(A7) hold and $(\phi_{h,\tau},q_{h,\tau},\bmu_{h,\tau})$ be a solution of Problem~\ref{prob:full}.  
Furthermore, let $(\hat \phi_{h,\tau},\hat q_{h,\tau},\hbmu_{h,\tau}) $ be given and $\bar r_{i,h,\tau}$, $i=1,2,3$, be the residuals defined by \eqref{eq:pgp1}--\eqref{eq:pgp3}. 
Then
\begin{align*}
&\E_\alpha(\phi_{h,\tau},q_{h,\tau}|\hat\phi_{h,\tau},\hat q_{h,\tau}) \big|_{t=t^n} + 
c' \int_0^{t^{n}} 
\|\bar \mu_{h,\tau}-\hbmu_{h,\tau}\|_1^2 
+ \|\bar q_{h,\tau} -\hbq_{h,\tau} \|_1^2 \ddta \\
&  \leq C_1' \E_\alpha(\phi_{h,\tau},q_{h,\tau}|\hat\phi_{h,\tau},\hat q_{h,\tau})\big|_{t=0} 
 + C_2' \int_0^{t^{n}} \|\bar r_{1,h,\tau}\|_{-1,h}^2 + \|\bar r_{2,h,\tau}\|_1^2 + \|\bar r_{3,h,\tau}\|_{-1,h}^2 \, \ddta
\end{align*}
for all $0 \le n \le N$ with constants $c'$, $C_1'$, $C_2'$ that are independent of the discretization parameters. 
\end{lemma}

\begin{proof}
The remainder of this section is devoted to the proof of this assertion. We start with splitting $\E_\alpha(\phi_{h,\tau},q_{h,\tau}|\hat \phi_{h,\tau},\hat q_{h,\tau}) = \E_\alpha^\phi(\phi_{h,\tau}|\hat \phi_{h,\tau})+\E_\alpha^q(q_{h,\tau}|\hat q_{h,\tau})$ and then estimate the change in the two parts of the relative energy over a single time interval separately. 

\subsection*{Bulk stress}
By the fundamental theorem of calculus, we obtain 
\begin{align*}
\E_\alpha^q(q_{h,\tau}|\hat q_{h,\tau}) \big|_{t^{n-1}}^{t^n} 
&=\la \dt q_{h,\tau} - \dt \hat q_{h,\tau},  q_{h,\tau} - \hat q_{h,\tau}\ran \\
&= \la \dt q_{h,\tau} - \dt \hat q_{h,\tau},  \bq_{h,\tau}-\hbq_{h,\tau}\ran = (*). 
\end{align*}
In the second step, we have used the definition of the orthogonal projection $\bar q = \bar \pi_\tau^0 q$.
Using the test function $\bar \zeta_{h,\tau}=\bq_{h,\tau}-\hbq_{h,\tau}$ in \eqref{eq:pg3} and \eqref{eq:pgp3}, we further obtain
\begin{align}
(*) = &- \la \kappa(\bar\phi_{h,\tau})(\bq_{h,\tau}-\hbq_{h,\tau}),\bq_{h,\tau}-\hbq_{h,\tau} \ran \label{eq:q_exp}\\
& - \eps \|\nabla (\bq_{h,\tau}-\hbq_{h,\tau})\|_{0}^2 + \la \bar r_{3,h,\tau},\bq_{h,\tau}-\hbq_{h,\tau}\ran \nonumber\\
&-\la d_0 \na(A(\bar\phi_{h,\tau})(\bq_{h,\tau}-\hbq_{h,\tau}))-c(\bar\phi_{h,\tau})\na(\bmu_{h,\tau}-\hbmu_{h,\tau}), \na(A(\bar\phi_{h,\tau})(\bq_{h,\tau}-\hbq_{h,\tau}))\ran.  \nonumber
\end{align}
The first two terms already appear in the dissipation functional. The last one will be abbreviated by $\mathcal{R}^n_q$ in the following and kept for later. 
The second term can be estimated by
\begin{align*}
\la\bar r_{3,h,\tau} ,\bq_{h,\tau}-\hbq_{h,\tau}\ran 
 \leq \delta \int_{I^n} \norm{\na(\bq_{h,\tau}-\hbq_{h,\tau})}_1^2 + C(\delta) \norm{\bar r_{3,h,\tau}}^2_{-1,h}\ddta.
\end{align*}
The parameter $\delta>0$ stems from the application of Young's inequality and will be chosen sufficiently small, but independent of the discretization parameters, to absorb the corresponding terms on the left-hand side. As a consequence, $C(\delta) \approx \delta^{-1}$ will only depend on the bounds in our assumptions.  
In summary, we thus obtain 
\begin{align}
&\E_\alpha^q(q_{h,\tau}|\hat q_{h,\tau})\big|_{t^{n-1}}^{t^n} \label{eq:disc_rel_en_q}
\leq -  c_1' \int_{I^n} \|\bq_{h,\tau}-\hbq_{h,\tau}\|_1^2 \ddta  + \mathcal{R}^n_q  + C_1' \int_{I_n} \norm{ \bar r_{3,h,\tau}}_{-1,h}^2 \ddta
\end{align}
with positive constants $c_1'$, $C_1'$ independent of the discretization parameters.
The first term on the right-hand side has a negative term and allows to compensate similar terms arising later on.

\subsection*{Cahn-Hilliard}
By the fundamental theorem of calculus, repeated application of the chain rule, the definition of the relative energy, and elementary computations, we obtain
\begin{align}
\E_\alpha^{\phi}((\phi_{h,\tau}&|\hat \phi_{h,\tau}) \big|_{t^{n-1}}^{t^n} 
= \gamma \la \na(\bar\phi_{h,\tau}-\hbphi_{h,\tau}),\na\dt(\phi_{h,\tau}-\hat\phi_{h,\tau}) \ran \notag\\
& + \la f'(\phi_{h,\tau})-f'(\hat\phi_{h,\tau}),\dt(\phi_{h,\tau}-\hat\phi_{h,\tau}) \ran + \alpha\la \bar\phi_{h,\tau}-\hbphi_{h,\tau},\dt(\phi_{h,\tau}-\hat\phi_{h,\tau}) \ra \notag\\
& + \la f'(\phi_{h,\tau})-f'(\hat\phi_{h,\tau})-f''(\hat\phi_{h,\tau}),\dt\hat\phi_{h,\tau} \ran 
 = I_1 + I_2 + I_3 + I_4.\label{eq:phi_rel_comp} 
\end{align}
\subsubsection*{Step~1.}
By testing \eqref{eq:pg2} and \eqref{eq:pgp2} with $\bar\xi_{h,\tau}=\dt(\phi_{h,\tau}-\hat\phi_{h,\tau})$, we obtain
\begin{align*}
I_1+I_2 & = \la \bmu_{h,\tau} - \hbmu_{h,\tau}+\bar r_{2,h,\tau},\dt(\phi_{h,\tau}-\hat\phi_{h,\tau}) \ran = (*).
\end{align*}
Next we use $\bar\psi_{h,\tau}=\bmu_{h,\tau}-\hbmu_{h,\tau}+\bar r_{2,h,\tau}$ as test function in \eqref{eq:pg1} and \eqref{eq:pgp1} to deduce
\begin{align*}
 &(*)= \la \bmu_{h,\tau} - \hbmu_{h,\tau}+\bar r_{2,h,\tau},\dt(\phi_{h,\tau}-\hat\phi_{h,\tau}) \ran \\
 &= -\la b(\bar\phi_{h,\tau})\na(\bmu_{h,\tau}-\hbmu_{h,\tau}) - c(\bar\phi_{h,\tau})\na(A(\bar\phi_{h,\tau})(\bq_{h,\tau}-\hbq_{h,\tau})),\na(\bmu_{h,\tau}-\hbmu_{h,\tau}+\bar r_{2,h,\tau}) \ran \\
 &\qquad \qquad \qquad + \la \bar r_{1,h,\tau},\bmu_{h,\tau}-\hbmu_{h,\tau}+\bar r_{2,h,\tau} \ran 
 = (i) + (ii).
\end{align*}
The first term can be combined with the remainder $\mathcal{R}^q_q$ in \eqref{eq:disc_rel_en_q}. By decomposition of the dissipative terms, similar as in the proof of the discrete energy-dissipation identity, estimating coefficient from below, and application of Young's inequality, we get 
\begin{align*}
\mathcal{R} + (i) 
&\leq \int_{I_n}  C'\norm{\bar r_{2,h,\tau}}_{1}^2  - c_2'  \|\nabla (\bar\mu_{h,\tau} - \hbmu_{h,\tau})\|_0^2 \\
&\qquad \qquad - c_3' \|c(\bar\phi_{h,\tau}) \nabla (\bmu_{h,\tau} - \hbmu_{h,\tau})- d_0 \na(A(\bar\phi_{h,\tau})(\bq_{h,\tau}-\hbq_{h,\tau}))\|_0^2 \ddta.
\end{align*}
The last two terms have a negative sign and will be used to compensate for terms of this form in the other estimates.
Using the definition of the discrete dual norm, Poincar\'e's inequality, and the bounds for the coefficients, the second term can be further estimated by
\begin{align*}
(ii) &\leq \int_{I_n}\norm{\bar r_{1,h,\tau}}_{-1}\left(\norm{\bmu_{h,\tau}-\hbmu_{h,\tau}}_1+\norm{\bar r_{2,h,\tau}}_1\right) \ddta \\
&\leq \int_{I_n} C_1' \norm{\bar r_{1,h,\tau}}_{-1,h}^2 + C_2'|\la \bmu_{h,\tau}-\hbmu_{h,\tau},1 \ra|^2 + \delta \, \|\nabla (\bmu_{h,\tau}-\hbmu_{h,\tau}) \|_0^2 +  C_3' \norm{\bar r_{2,h,\tau}}_1^2 \ddta.
\end{align*}
The parameter $\delta>0$ will again be chosen sufficiently small, but independent of the discretization parameters. 
By testing the variational identities \eqref{eq:pg2} and \eqref{eq:pgp2} with $\bar\xi_{h,\tau}=1$, we see that
\begin{align} \label{eq:mu_l2_bound}
   \int_{I_n} |\la \bmu_{h,\tau}-\hbmu_{h,\tau},1 \ra|^2 \ddta
    &= \int_{I_n}|\la f'(\phi_{h,\tau}) -f'(\hat\phi_{h,\tau}) +\bar r_{2,h,\tau},1\ra|^2 \ddta \\
    &\leq \int_{I_n} C_4' \norm{\bar r_{2,h,\tau}}_{0,1}^2 + C_5' \|\phi_{h,\tau} - \hat \phi_{h,\tau}\|_1^2 \ddta. \notag
\end{align}
The constants $C_4',C_5'$ depend on the bounds in the assumptions and the uniform a-priori bounds for the discrete solution established in Theorem~\ref{thm:ex}. 
In summary, we can thus estimate 
\begin{align*}
(ii) \le \int_{I_n}  \delta \, \|\nabla (\bmu_{h,\tau}-\hbmu_{h,\tau}) \|_0^2 + C_6' \norm{\bar r_{1,h,\tau}}_{-1,h}^2 + C_7' \norm{\bar r_{2,h,\tau}}_1^2 + C_8'  \E^\phi_\alpha(\phi_{h,\tau}|\hat \phi_{h,\tau}) \ddta.
\end{align*}
The first term can later be absorbed by dissipation terms and choosing $\delta$ appropriately.

\subsubsection*{Step~2.}
By testing the variational identities \eqref{eq:pg1} and \eqref{eq:pgp1} with $\bar \psi_{h,\tau}=\alpha(\bar\phi_{h,\tau}- \hat{\bar\phi}_{h,\tau})$, we get 
\begin{align*}
I_3 &= \alpha \la \bar\phi_{h,\tau}-\hbphi_{h,\tau},\dt(\phi_{h,\tau}-\hat\phi_{h,\tau})\ran 
= \alpha \la \bar r_{1,h,\tau}, \bar \phi_{h,\tau}-\hat{\bar\phi}_{h,\tau}\ran \\
&\qquad  -\alpha  \la b(\bar\phi_{h,\tau})\na(\bmu_{h,\tau}-\hbmu_{h,\tau}) - c(\bar\phi_{h,\tau})\na(A(\bar\phi_{h,\tau})(\bq_{h,\tau}-\hbq_{h,\tau})), \na(\bar\phi_{h,\tau}-\hbphi_{h,\tau})\ran. \end{align*}
With similar arguments as before, we then obtain
\begin{align*}
I_3 &\leq \int_{I_n}  C_1' \norm{\bar r_{1,h,\tau}}_{-1,h}^2 + C_2' \norm{\phi_{h,\tau}-\hat\phi_{h,\tau}}_1^2 + C_3' \delta \|\nabla (\bmu_{h,\tau}-\hbmu_{h,\tau}) \|_0^2 \\
& \qquad \qquad + C_4' \delta \norm{c(\bar\phi_{h,\tau})\na(\bmu_{h,\tau}-\hbmu_{h,\tau}) - d_0 \na(A(\bar\phi_{h,\tau})(\bq_{h,\tau}-\hbq_{h,\tau}))}_0^2 \ddta.  
\end{align*}
For $\delta$ sufficiently small, the corresponding term can again be absorbed by dissipation terms.
\subsubsection*{Step~3.}
From the bounds in assumption (A3), we can deduce that 
\begin{align*}
|f'(\phi_{h,\tau}) - f'(\hat \phi_{h,\tau}) - f''(\hat \phi_{h,\tau})(\phi_{h,\tau} - \hat \phi_{h,\tau})| 
&\le \left(f_2^{(3)} + f_3^{(3)}(|\phi_{h,\tau}| + |\hat \phi_{h,\tau}|) \right) |\phi_{h,\tau} - \hat \phi_{h,\tau}|^2.
\end{align*}
Using the H\"older inequality, embedding estimates, and the uniform bounds for $\phi_{h,\tau}$ \eqref{eq:full_apriori}, we can further bound $I_4$ from above as follows
\begin{align*}
I_4 &\le \int_{I_n}\|\dt \hat \phi_{h,\tau}\|_{0} \|f'(\phi_{h,\tau}) - f'(\hat \phi_{h,\tau}) - f''(\hat \phi_{h,\tau})(\phi_{h,\tau} - \hat \phi_{h,\tau})\|_{0} \ddta \\
&\le C_1' \, \int_{I^n} \|\phi_{h,\tau} - \hat \phi_{h,\tau}\|_{0,6}^2 \ddta 
\le C_2' \, \int_{I^n} \E^\phi_\alpha(\phi_{h,\tau}|\hat \phi_{h,\tau}) \ddta.
\end{align*}
The constants $C_1'$, $C_2'$ only depend on the bounds of the coefficients and available uniform bounds for the discrete and exact solutions.

\subsection*{Stability estimate.}
Combining all bounds derived so far, and using \eqref{eq:mu_l2_bound} once again in order to bound $\|\bar \mu_{h,\tau} - \hbmu_{h,\tau}\|^2_0$, we find the following inequality
\begin{align}
&\E_\alpha(\phi_{h,\tau},q_{h,\tau}|\hat\phi_{h,\tau},\hat q_{h,\tau})\big|_{t^{n-1}}^{t^n} 
+ c' \int_{I_n} \|\bmu_{h,\tau}-\hbmu_{h,\tau}\|_1^2 + \|\bar q_{h,\tau}-\hbq_{h,\tau})\|_1^2 \ddta\label{eq:full_dsic_rel_prelim1}\\
&\leq 
    C_1' \int_{I_n} \E_\alpha(\phi_{h,\tau},q_{h,\tau}|\hat\phi_{h,\tau},\hat q_{h,\tau}) \ddta 
    + C_2' \int_{I_n} \norm{\bar r_{1,h,\tau}}_{-1,h}^2 + \norm{\bar r_{2,h,\tau}}_1^2 + \norm{\bar r_{3,h,\tau}}_{-1,h} \ddta.  \nonumber
\end{align}
We note that the constants $c'$, $C_1'$, $C_2'$ only depend on available bounds for the discrete and continuous solution and for the parameters. 
The estimate of Lemma~\ref{lem:fullstab} then follows by recursive application of this inequality and the discrete Gronwall lemma \cite{Wloka}.
\end{proof}

\section{Error estimates}
\label{sec:error}

In order to prove the global error estimates of Theorem~\ref{thm:full_error_est}, we define
\begin{align} \label{eq:fullproj_1}
\hat \phi_{h,\tau} = \I_\tau^1 \pi_h^1 \phi 
\quad \text{and} \quad 
\hbmu_{h,\tau} &= \bar \pi_\tau^0 \pi_h^0 \mu \quad \text{ and} \quad \hat q_{h,\tau}=\I_\tau^1\pi_h^0 q 
\end{align}
as approximations for the continuous solution. We can then split the error into a projection error and a discrete evolution error, using standard error estimates for the first, and the discrete stability results of the previous section to bound the second error component. 

\subsection*{Projection errors.}
By standard estimates of polynomial interpolation and projection errors, we obtain the following estimates; see \cite{BrennerScott} and the appendix. 
\begin{lemma}\label{lem:projerr}
Let (A6)--(A7) hold and $(\hat\phi_{h,\tau},\hbmu_{h,\tau},\hat q_{h,\tau})$ be given as above. Let $(\phi,\mu,q)$ be a smooth solution of \eqref{eq:chq1}--\eqref{eq:chq3} such that \eqref{eq:reg1}--\eqref{eq:reg3} holds.  Then
\begin{align*}
&\norm{\hat\phi_{h,\tau} - \phi}_{L^\infty(H^1)}^2 \leq C(h^4 + \tau^4),& &\norm{\hbmu_{h,\tau}-\bmu}_{L^2(H^1)}^2 \leq Ch^4, \\
&\norm{\hat q_{h,\tau} - q}_{L^\infty(L^2)}^2 \leq C(h^4 + \tau^4),& &\norm{\hbq_{h,\tau} - \bar q}_{L^2(H^1)}^2 \leq Ch^4, \\
&\|\dt \hat\phi_{h,\tau}-\bar \pi_\tau^0 (\dt \phi)\|^2_{L^2(H^{-1})} \leq C h^4,&
&\|\dt\hat q_{h,\tau}-\bar \pi_\tau^0 (\dt q)\|^2_{L^2(H^{-1})} \leq Ch^4.
\end{align*}
The constant $C$ in these estimates depends only on a-priori bounds for the solution components in appropriate norms and constants in the assumptions.
\end{lemma}

\subsection*{Residuals}
In the next step, we identify and then estimate the residuals arising from the above choice of approximate functions and \eqref{eq:pgp1}--\eqref{eq:pgp3}.
\begin{lemma}
Let the assumptions of Theorem~\ref{thm:full_error_est} hold and $(\hat \phi_{h,\tau},\hbmu_{h,\tau},\hat q_{h,\tau})$ be defined as projections of the smooth solution $(\phi,\mu,q)$ via \eqref{eq:fullproj_1}. Then \eqref{eq:pg1}--\eqref{eq:pg3} holds with the residuals $\bar r_{i,h,\tau}$ defined by
\begin{align*}
\la \bar r_{1,h,\tau}, \bar \psi_{h,\tau}\ran 
&= \la \dt (\pi_h^1 \phi - \phi), \bar \psi_{h,\tau}\ra + \la b(\bar\phi_{h,\tau})\na\hbmu_{h,\tau} -c(\bar\phi_{h,\tau})\na(A(\bar\phi_{h,\tau})\hbq_{h,\tau}), \nabla \bar\psi_{h,\tau}\ran \\
& \qquad \qquad- \la b(\phi)\na\mu - c(\phi) \na(A(\phi)q)) , \nabla \bar\psi_{h,\tau}\ran,\\
\la  \bar r_{2,h,\tau}, \bar \xi_{h,\tau} \ran
&= \la \hbmu_{h,\tau} - I_\tau^1 \mu, \bar \xi_{h,\tau}\ran 
      + \gamma \la \nabla (\hbphi_{h,\tau} - \I_\tau^1 \phi), \nabla \bar \xi_{h,\tau}\ran  \\
      &+ \la f'(\hat \phi_{h,\tau}) - \I_\tau^1 f'(\phi), \bar \xi_{h,\tau}\ran ,\\
\la  \bar r_{3,h,\tau},\bar \zeta_{h,\tau} \ran 
&=  \la \kappa(\bar\phi_{h,\tau})\hbq_{h,\tau} - \kappa(\phi)\na q,\bar\zeta_{h,\tau} \ran  
+ \varepsilon \la \na\hbq_{h,\tau} - \na q,\na\bar\zeta_{h,\tau} \ran  \\
& \qquad \qquad + \la d_0 \na(A(\bar\phi_{h,\tau})\hbq_{h,\tau})-c(\bar\phi_{h,\tau})\na\hbmu_{h,\tau},\na(A(\bar\phi_{h,\tau})\bar\zeta_{h,\tau})\ran \\
& \qquad\qquad -\la d_0 \na(A(\phi)q) - c(\phi)\na\mu, \na(A(\phi)\bar\zeta_{h,\tau})\ran.
\end{align*}
\end{lemma}
\begin{proof}
The identities follow directly from the variational principles characterizing the exact smooth solution and its discrete projection, and some elementary manipulations.
\end{proof}

Using interpolation and projection error estimates one can derive the following bounds for the residuals after some tedious calculations. 
\begin{lemma}\label{lem:est_res_full}
Under the assumptions of the previous lemma, we have 
\begin{align*}
 \int_{I_n} \norm{\bar  r_{1,h,\tau}}_{-1,h}^2 
 + \norm{\bar  r_{2,h,\tau}}_{1}^2 
 + \norm{\bar  r_{3,h,\tau}}_{-1,h}^2 \ddta 
 &\le C_1' (h^4 + \tau^4) \\
 &+ C_2' \int_{I_n} \E_\alpha(\phi_{h,\tau},q_{h,\tau}|\hat \phi_{h,\tau},\hat q_{h,\tau}) \ddta.
\end{align*}
The constants $C_1'$, $C_2'$ are independent of the discretization parameters.
\end{lemma}
\noindent 
The detailed proof of this assertion will be given in the appendix.

\subsection*{Discrete error}
By combination of the previous results, we can 
now prove the following bounds for the discrete evolution error.
\begin{lemma}\label{lem:disc_err_pg}
Let the assumptions of Theorem~\ref{thm:full_error_est} hold.
Then
\begin{align*}
\norm{\phi_{h,\tau}-\hat\phi_{h,\tau}}_{L^\infty(H^1)}^2 &+  \norm{q_{h,\tau}-\hat q_{h,\tau}}_{L^\infty(L^2)}^2   \\
&+ \norm{\bmu_{h,\tau}-\hbmu_{h,\tau}}_{L^2(H^1)}^2
+ \norm{\bq_{h,\tau}-\hbq_{h,\tau}}_{L^2(H^1)}^2   \leq C(h^4 + \tau^4)  
\end{align*}
with a constant $C$ that is independent of the discretization parameters $h$ and $\tau$.
\end{lemma}
\begin{proof}
By using Lemma~\ref{lem:fullstab}, the bounds of Lemma~\ref{lem:est_res_full}, and applying the discrete Gronwall lemma, similar to the proof of Lemma~\ref{lem:fullstab}, we obtain
\begin{align*}
&\E_{\alpha}(\phi_{h,\tau},q_{h,\tau}|\hat \phi_{h,\tau},\hat q_{h,\tau})\big|_{t=t^n} + c' \int_0^{t^n} \|\bar \mu_{h,\tau} - \hbmu_{h,\tau}\|_1^2 + \|\bar q_{h,\tau} - \hbq_{h,\tau}\|_1^2 \ddta \\
& \qquad \qquad \qquad 
\le C_1 \E_{\alpha}(\phi_{h,\tau},q_{h,\tau}|\hat \phi_{h,\tau},\hat q_{h,\tau})\big|_{t=0}
+ C(h^4 + \tau^4) \qquad \text{for all } 0 \le n \le N.
\end{align*}
Due to the particular choice of the initial values $\phi_{h,0}$, $q_{h,0}$, the first term on the right-hand side drops out. 
Using \eqref{eq:lower_bound_rel}, the first term on the left-hand side can further be estimated from below by the corresponding norms, which already yields the result. 
\end{proof}

\subsection*{Conclusion}
By combination of the estimates for the projection errors and the discrete evolution error, we can finally obtain the global error estimates of Theorem~\ref{thm:full_error_est}. 

\section{Uniqueness of discrete solutions}
\label{sec:uniqueness}

We will now use the discrete stability results of Lemma~\ref{lem:fullstab} to show that, under a mild restriction on the time step $\tau$, we can also expect the uniqueness of the discrete solution. 
\begin{lemma} \label{lem:res_unique}
Let $(\phi_{h,\tau},\bar \mu_{h,\tau},q_{h,\tau})$ and $(\hat\phi_{h,\tau},\hbmu_{h,\tau},\hat q_{h,\tau})$ be two solutions of Problem~\ref{prob:full} with the same initial data. Then 
$(\hat\phi_{h,\tau},\hbmu_{h,\tau},\hat q_{h,\tau})$ satisfies \eqref{eq:pgp1}--\eqref{eq:pgp3} with residuals $\bar r_{2,h,\tau}=0$ and
\begin{align*}
\la \bar r_{1,h,\tau},\bar \psi_{h,\tau}\ran 
&= \la b(\bar\phi_{h,\tau}) \nabla \hbmu_{h,\tau}-c(\bar \phi_{h,\tau}) \nabla (A(\bar \phi_{h,\tau}) \hbq_{h,\tau}),\na\bar\psi_{h,\tau}\ran \\
& \qquad -\la b(\hbphi_{h,\tau}) \nabla \hbmu_{h,\tau}-c(\hbphi_{h,\tau}) \nabla (A(\hbphi_{h,\tau}) \hbq_{h,\tau}),\na\bar\psi_{h,\tau}\ran \\
\la \bar r_{3,h,\tau},\bar \psi_{h,\tau}\ran 
&= \la d_0 \nabla (A(\bar \phi_{h,\tau}) \hbq_{h,\tau}) - c(\bar \phi_{h,\tau}) \nabla \hbmu_{h,\tau}, \nabla (A(\bar \phi_{h,\tau}) \bar \zeta_{h,\tau})\ran\\
& \qquad - \la d_0 \nabla (A(\hbphi_{h,\tau}) \hbq_{h,\tau}) - c(\hbphi_{h,\tau}) \nabla \hbmu_{h,\tau}, \nabla (A(\hbphi_{h,\tau}) \bar \zeta_{h,\tau})\ran \\
& \qquad + \la \kappa(\bar \phi_{h,\tau}) \hbq_{h,\tau}, \bar \zeta_{h,\tau} \ran - \la \kappa(\hbphi_{h,\tau})  \hbq_{h,\tau}, \bar \zeta_{h,\tau}\ran.
\end{align*}
\end{lemma}
The assertions again follow immediately from the definition of the discrete solutions. 
Similar as above, we state appropriate bounds for the residuals in the relevant norms.
\begin{lemma} \label{lem:unique_res}
Let (A0)--(A7) hold and $(\phi_{h,\tau},\mu_{h,\tau},q_{h,\tau})$ respectively $(\hat \phi_{h,\tau}, \hbmu_{h,\tau},\hat q_{h,\tau})$ denote two solutions of Problem~\ref{prob:full}. Then the residuals of Lemma~\ref{lem:res_unique} can be estimated by
\begin{align*}
\int_{I_n} \norm{\bar r_{1,h,\tau}}_{-1,h}^2 + \norm{\bar r_{3,h,\tau}}_{-1,h}^2 \ddta 
\leq  C'' \int_{I_n}\E_\alpha(\phi_{h,\tau},q_{h,\tau}|\hat \phi_{h,\tau},\hat q_{h,\tau}) \ddta. 
\end{align*}
The constant $C''$ in this estimate depends on bounds of the model parameters, the initial data, and additionally on bounds of $\norm{\hbmu_{h,\tau}}_{L^\infty(W^{1,3})}$, $\norm{\hbq_{h,\tau}}_{L^\infty(W^{1,3})}$, $\norm{\hbq_{h,\tau}}_{L^\infty(L^\infty)}$, and $\norm{\hbphi_{h,\tau}}_{L^\infty(W^{1,3})}$.
\end{lemma}
\noindent 
The detailed proof of these estimates is provided in the appendix.

\subsection*{Uniqueness}

In order to proceed, we need to estimate the norms of the discrete solution mentioned at the end of Lemma~\ref{lem:unique_res} uniformly in the discretization parameters.  
To this end, we use the following arguments: Let $\bar g_{h,\tau}$ stand for $\bar \mu_{h,\tau}$ or $\bar q_{h,\tau}$. Then 
\begin{equation*}
   \norm{\bar g_{h,\tau}}_{L^\infty(W^{1,3})} \leq \norm{\bar g_{h,\tau} -  \pi_h^0 \bar g}_{L^\infty(W^{1,3})} + \norm{\pi_h \bar g - \bar g}_{L^\infty(W^{1,3})} + \norm{\bar g}_{L^\infty(W^{1,3})},
\end{equation*}
where $\pi_h^0$ is the $L^2$-projection on the space $\Vh$. The second and third terms can be uniformly bounded by projection error estimates assuming sufficient spatial regularity of the function $\bar g$; see Lemma~\ref{lem:projerr}. 
To bound the first term, we use the inverse inequality~\eqref{eq:inverse} with $p=3,q=2,d\le3$ in space and with $p=\infty,q=2,d=1$ in time, 
and the estimates of Lemma~\ref{lem:disc_err_pg}. This leads to
\begin{equation*}
\norm{\bar g_{h,\tau} - \pi_h^0 \bar g}_{L^\infty(W^{1,3})} \leq Ch^{-1/2}\tau^{-1/2}(h^4 + \tau^4).    
\end{equation*}
Similar estimates can be done for the $L^\infty(L^\infty)$ norm of $\hbq_{h,\tau}$.
For $\tau=ch$, the constant $C''$ in Lemma~\ref{lem:unique_res} can thus be bounded uniformly in $h$ and $\tau$. 
We can then apply the discrete Gronwall lemma, similarly as in the proof of Lemma~\ref{lem:fullstab}, to obtain
\begin{align*}
\E_{\alpha}(\phi_{h,\tau},q_{h,\tau}|\hat \phi_{h,\tau},\hat q_{h,\tau}) \big|_{t=t^n} &+ c''\int_0^{t^n} \|\bmu_{h,\tau}-\hbmu_{h,\tau}\|_1^2 + \|\bar q_{h,\tau}-\hbq_{h,\tau}\|_1^2   \ddta 
\le 0. 
\end{align*}
In the last step, we used that both functions $(\phi_{h,\tau},q_{h,\tau})$ as well as $(\hat\phi_{h,\tau},\hat q_{h,\tau})$ have the same initial values. 
Together with the lower bounds for the relative energy, we find that the difference of the two solutions vanishes.

\section{Numerical illustration}
\label{sec:num}

We complement the theoretical results by two computational tests. 
We consider the domain $\Omega=(0,1)^2$, which can be extended periodically to the whole of $\R^2$. Functions on $\Omega$ are assumed periodically extendable under preservation of class.

\subsection{Convergence rates}
We start with evaluating the convergence rate of the proposed method. 
We choose smooth initial data 
\begin{align*}
\phi_0 &= 0.25\cos(2\pi x)\cos(2\pi y) + 0.5, \qquad q_0 = 0.01\sin(2\pi x)\sin(2\pi y).
\end{align*}
On the time interval $[0,T], T=10$ the viscoelastic phase separation system \eqref{eq:chq1}--\eqref{eq:chq3} is expected to possesses a smooth exact solution.
The parameters are set to $\gamma=\varepsilon=10^{-3}$.
For the nonlinear functions, we chose $b(\phi) = c(\phi)^2 + \eps$, $c(\phi)=\frac{4}{\sqrt{10}}\cdot\phi (1-\phi)$, $d_0=1$,  $f(\phi)=16(\phi-0.95)^2(\phi-0.05)^2$, $\kappa(\phi) = 10^{-2}(10\phi^2 + 10^{-4})^{-1}$, and 
\begin{equation*}
A= 5 \cdot 10^{-3} \cdot \big[1+ \tanh(5[\cot(\pi\phi^*)-\cot(\pi\phi)])\big]
\end{equation*}
with $\phi^*=0.5=\la \phi_0, 1\ra$ denoting the total mass. 
Apart from the specific choice of the initial data, the problem setting is similar to that of \cite{Strasser.2019}. 
The discretization errors are estimated by comparing the computed solutions on two consecutively refined grids. 
The error quantities which we report in the following are defined as 
\begin{align*}
e_{h,\tau} &= \norm*{\phi_{h,\tau} - \phi_{h/2,\tau/2}}_{L^\infty(H^1)}^2 + \norm*{q_{h,\tau} - q_{h/2,\tau/2}}_{L^\infty(L^2)}^2 \\
&\qquad\quad+ \norm*{\bmu_{h,\tau} - \bmu_{h/2,\tau/2}}_{L^2(H^1)}^2 + \norm*{\bq_{h,\tau} - \bq_{h/2,\tau/2}}_{L^2(H^1)}^2. 
\end{align*}
These are the natural norms arising in the stability and error analysis of the problem.
We also present the individual error components, which are denoted by $e_{\phi,h,\tau},e_{q,h,\tau},e_{\bmu,h,\tau},e_{\bq,h,\tau}$.
In the Tables~\ref{tab:rates_time_chq1}--\ref{tab:rates_time_chq2}, we display the results of our computations obtained on a sequence of uniformly refined meshes with mesh size $h_k=2^{-(1+k)}$, $k=0,\ldots,6$, and time steps $\tau_k =  h_k$. 
\begin{table}[htbp!]
\centering
\small
\begin{tabular}{c||c|c|c|c|c|c}
$ k $ & $ e_{h,\tau} $  &  eoc & $ e_{\phi,h,\tau} $  &  eoc & $ e_{q,h,\tau} $  &  eoc    \\
\hline
$ 0 $   & $9.84 \cdot 10^{-1}$  & ---  & $8.78 \cdot 10^{-1}$ & --- &  $1.18 \cdot 10^{-10}$ & ---  \\
$ 1 $   & $6.55 \cdot 10^{-2}$  & 3.03 & $5.95 \cdot 10^{-2}$ & 3.88 & $1.39 \cdot 10^{-11}$ & 3.10 \\
$ 2 $   & $8.02 \cdot 10^{-3}$  & 3.03 & $7.64 \cdot 10^{-3}$ & 2.96 & $1.28 \cdot 10^{-12}$ & 3.62 \\
$ 3 $   & $5.32 \cdot 10^{-4}$  & 3.91 & $5.08 \cdot 10^{-4}$ & 3.91 & $2.47 \cdot 10^{-14}$ & 5.51 \\
$ 4 $   & $3.37 \cdot 10^{-5}$  & 3.98 & $3.22 \cdot 10^{-5}$ & 3.98 & $3.10 \cdot 10^{-16}$ & 6.32 
\end{tabular}
\bigskip 
\caption{Errors and experimental orders of convergence for smooth solution: Part I. \label{tab:rates_time_chq1}} 
\end{table}
\begin{table}[htbp!]
\centering
\small
\begin{tabular}{c||c|c|c|c}
$ k $ &  $ e_{\bmu,h,\tau} $  &  eoc & $ e_{\bq,h,\tau} $  &  eoc   \\
\hline
$ 0 $   & $1.04 \cdot 10^{-1}$ & ---  & $1.17 \cdot 10^{-6}$  & ---\\
$ 1 $   & $5.93 \cdot 10^{-3}$ & 4.14 & $4.04 \cdot 10^{-7}$  & 1.54\\
$ 2 $   & $3.78 \cdot 10^{-4}$ & 3.97 & $1.10 \cdot 10^{-7}$  & 1.87\\
$ 3 $   & $2.41 \cdot 10^{-5}$ & 3.97 & $1.13 \cdot 10^{-8}$  & 3.28\\
$ 4 $   & $1.52 \cdot 10^{-6}$ & 3.99 & $5.82 \cdot 10^{-10}$ & 4.28   
\end{tabular}
\bigskip 
\caption{Errors and experimental orders of convergence for smooth solution: Part II. \label{tab:rates_time_chq2}} 
\end{table}

As predicted by Theorem~\ref{thm:full_error_est}, we observe convergence of at least fourth order for the squared norms of the errors in all solution components. 
Let us finally mention that, as predicted, the discrete identities for the conservation of mass and energy dissipation are valid in our computations up to round-off errors.

\subsection{Qualitative behavior}\label{exp:visco}
This experiment illustrates typical features associated with viscoelastic phase separation. 
Similarly to~\cite{Strasser.2019,Diss},
we choose the initial data as
\begin{align*}
   \phi_0 = 0.4 + \xi(x,y), \qquad   q_0 = 0
\end{align*}
 with $\xi(x,y)$ a uniform random perturbation of small amplitude, i.e. $\xi(x,y)\in[-0.0025,0.0025]$. 
The model parameters are set to $\gamma=\varepsilon=10^{-3}$ and $T=12$. 
For the nonlinear functions, we chose $b(\phi) = c(\phi)^2 + \eps$, $c(\phi)=\frac{1}{\sqrt{10}} \phi(1-\phi)$, $d_0=1$, $f(\phi)=(\phi-0.95)^2(\phi-0.05)^2$, $\kappa(\phi) = 10^{-3}(10\phi^2 + 10^{-4})^{-1}$, and 
\begin{equation*}
A= \frac{1}{2}\big[1 + \tanh(10[\cot(\pi\phi^*)-\cot(\pi\phi)])\big]
\end{equation*}   
with $\phi^*=\la \phi_0,1 \ra$ again denoting the total mass.
\begin{figure}[htbp!]
\centering
\footnotesize
\begin{tabular}{ccc}
    \includegraphics[trim={3cm 1.4cm 1.1cm 0.5cm},clip,scale=0.29]{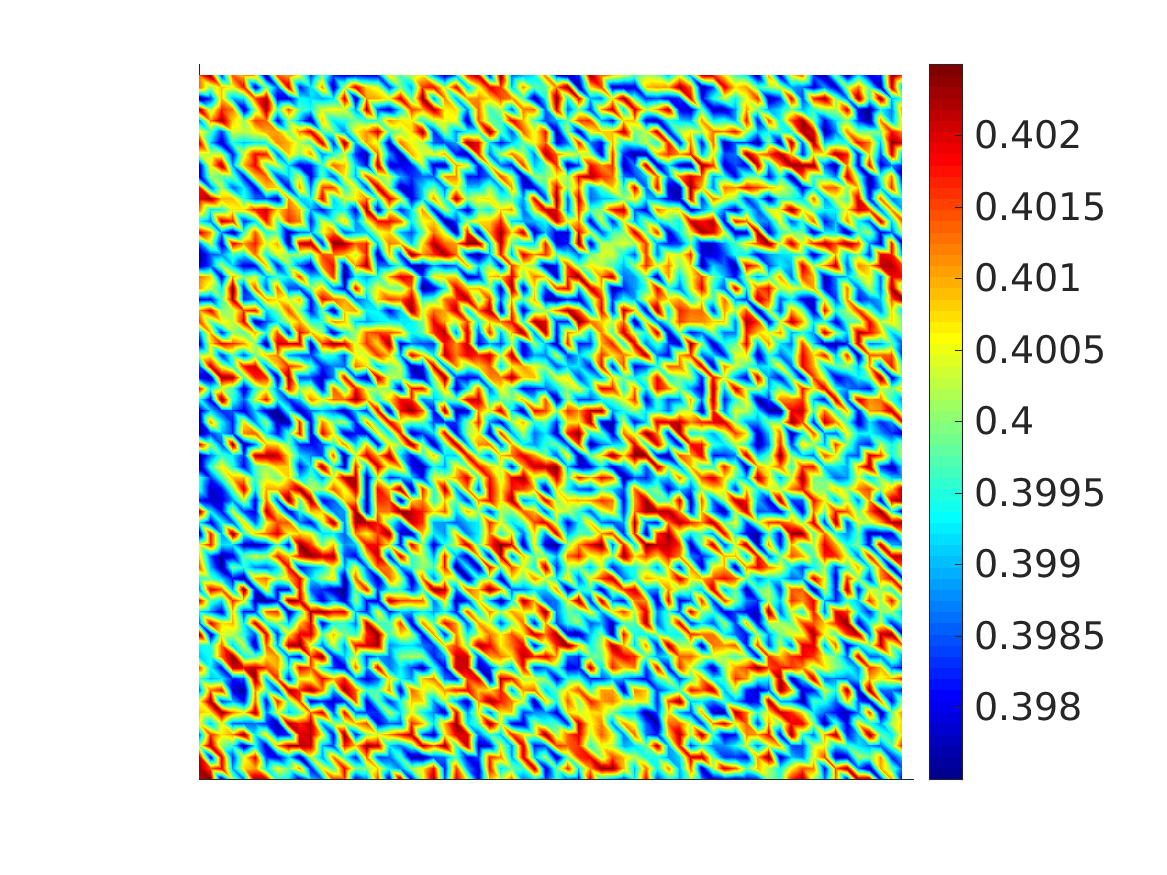} 
    &
    \includegraphics[trim={3.3cm 1.4cm 1.1cm 0.5cm},clip,scale=0.29]{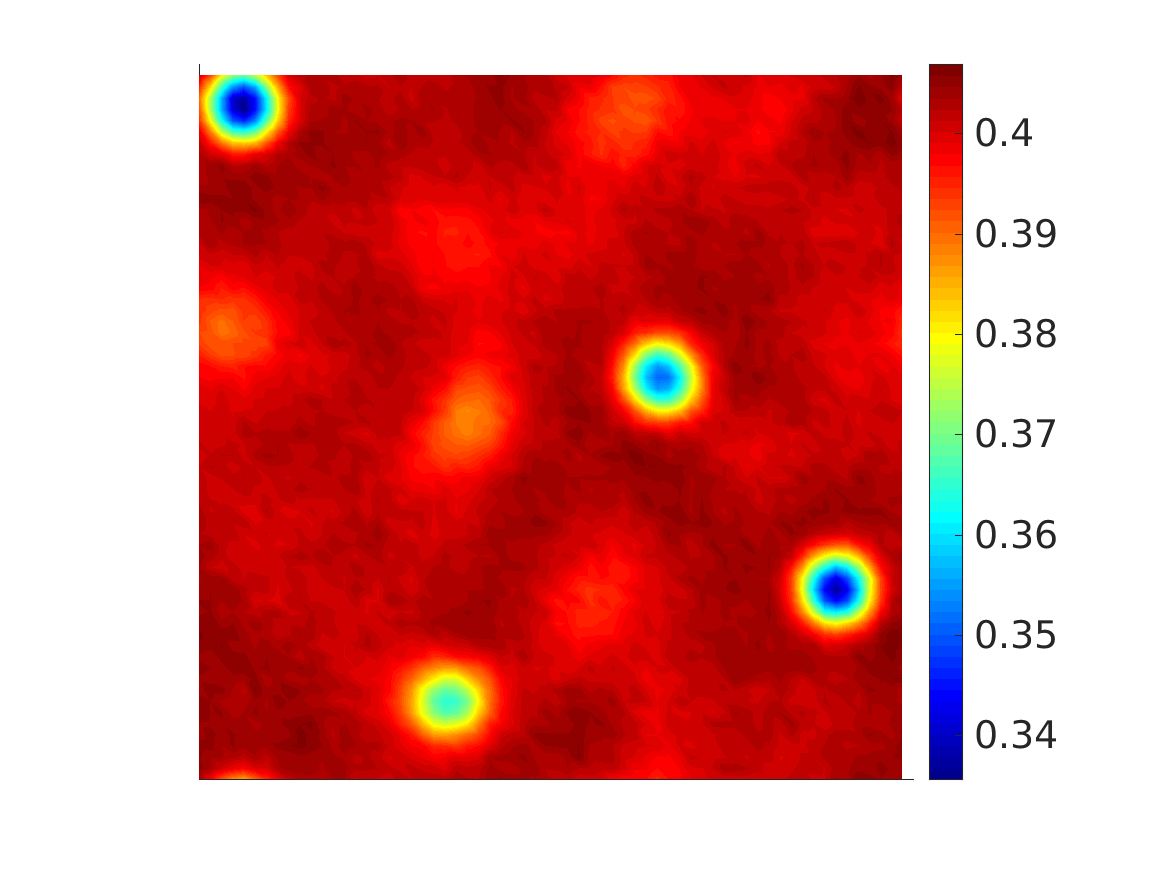}  
    &
    \includegraphics[trim={3.3cm 1.4cm 1.1cm 0.5cm},clip,scale=0.29]{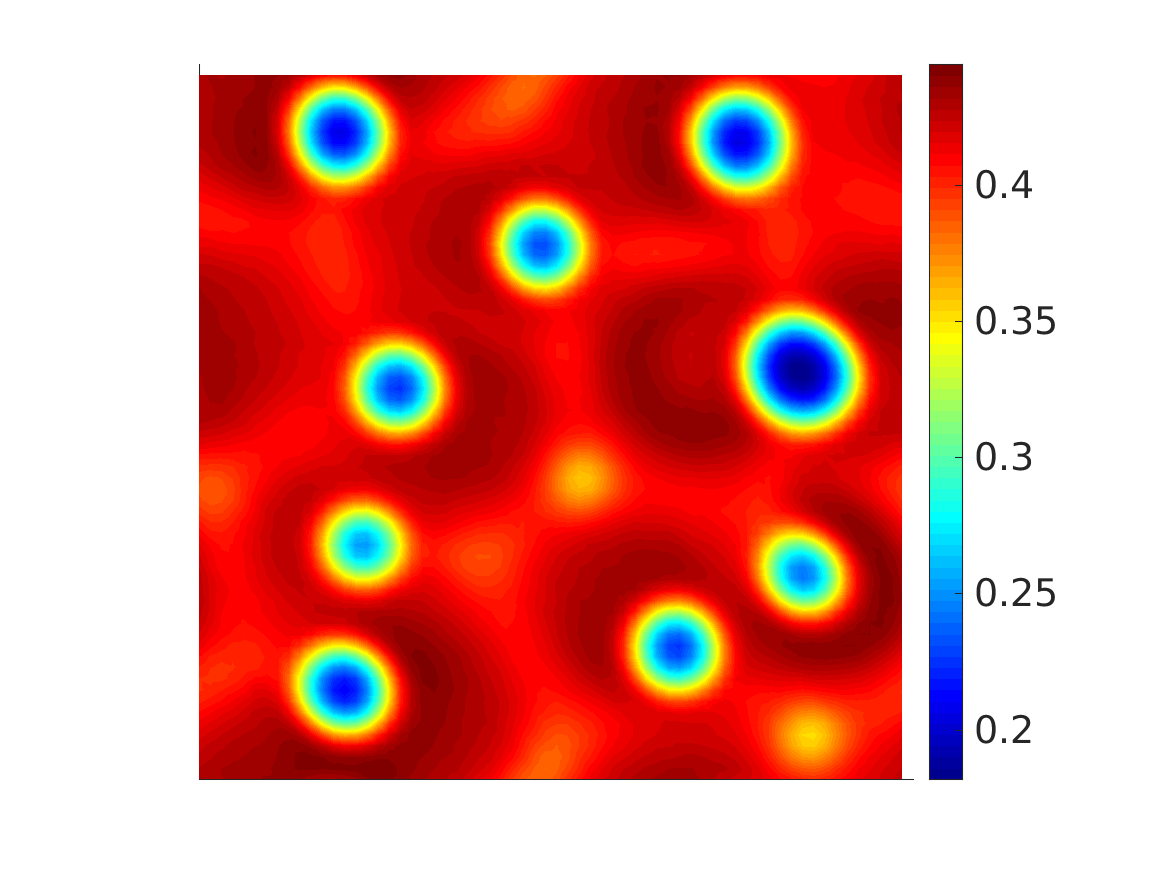}  \\
    \hspace{-1cm}t=0 & \hspace{-1cm}t=0.6 & \hspace{-1cm}t=1.7 \\
    \includegraphics[trim={3cm 1.4cm 1.1cm 0.5cm},clip,scale=0.29]{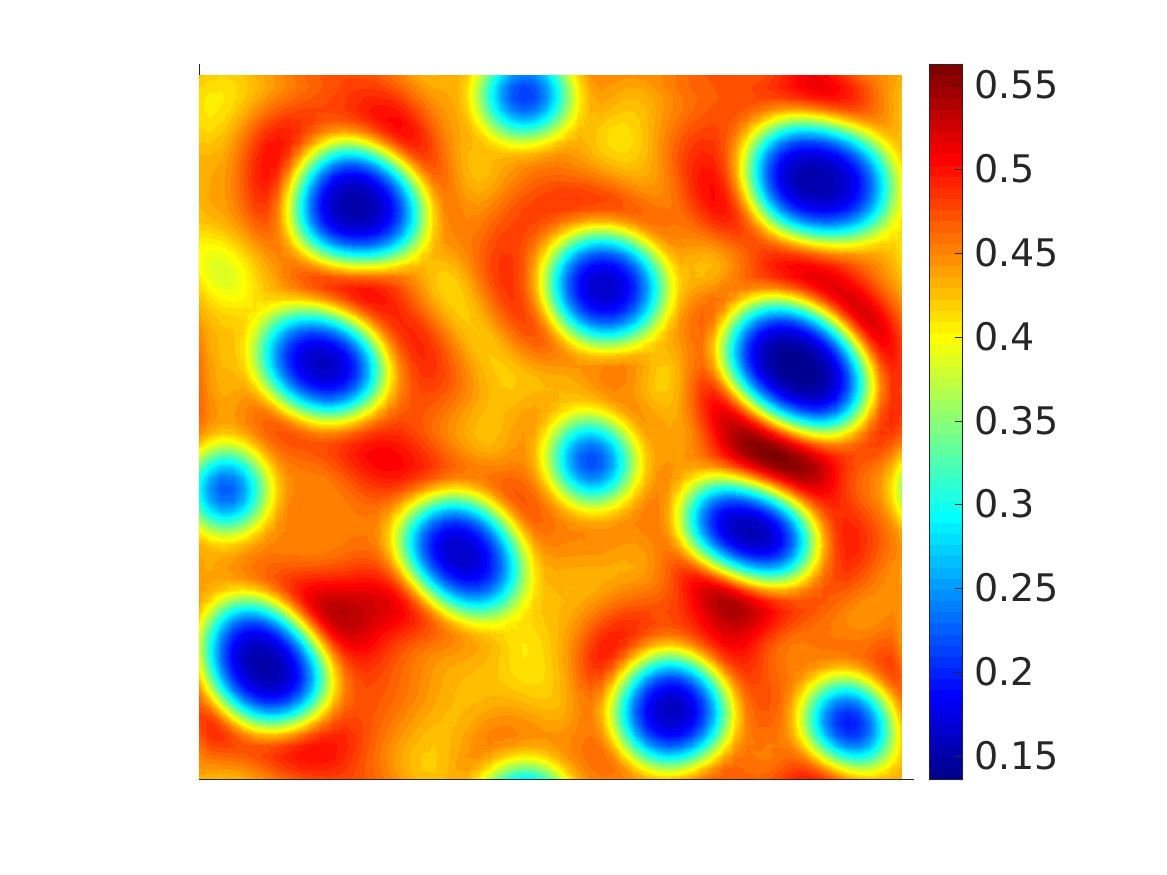}  
    &
    \includegraphics[trim={3.3cm 1.4cm 1.1cm 0.5cm},clip,scale=0.29]{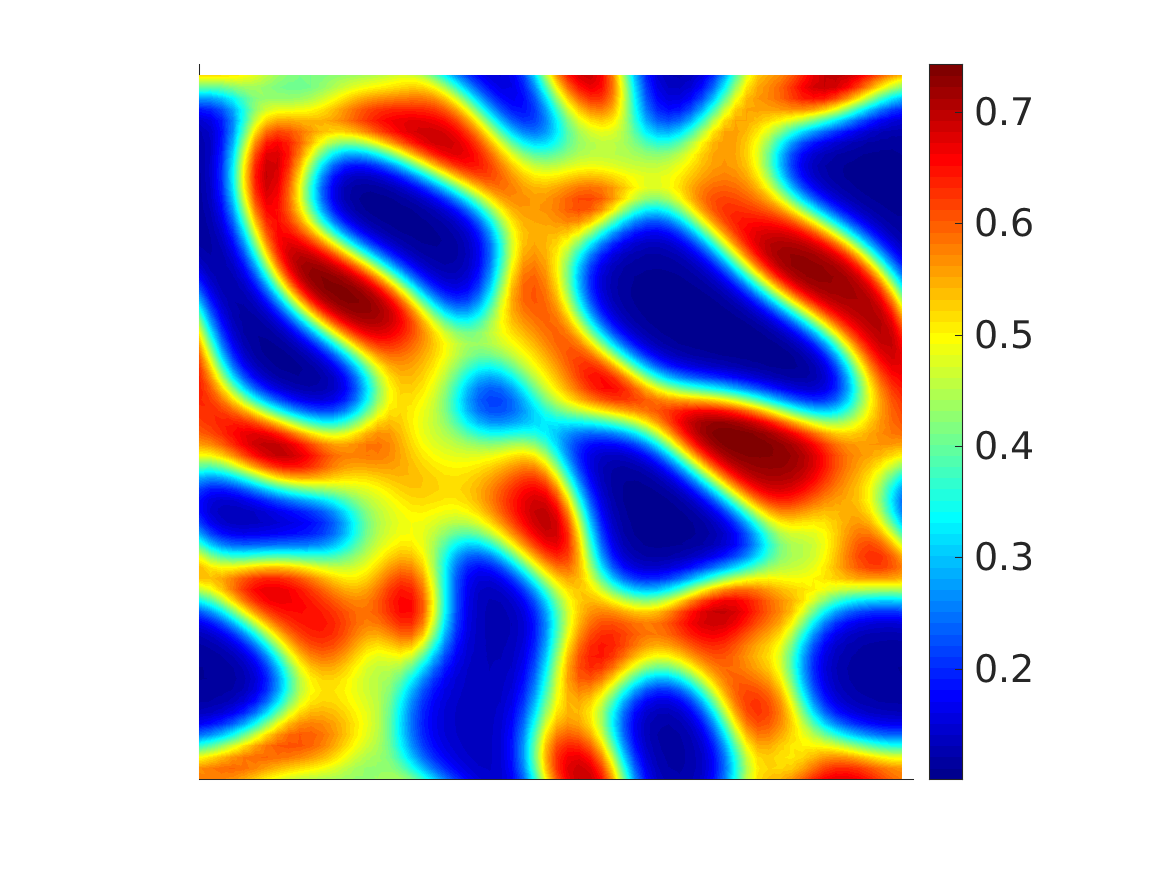} 
    &
    \includegraphics[trim={3.3cm 1.4cm 1.1cm 0.5cm},clip,scale=0.29]{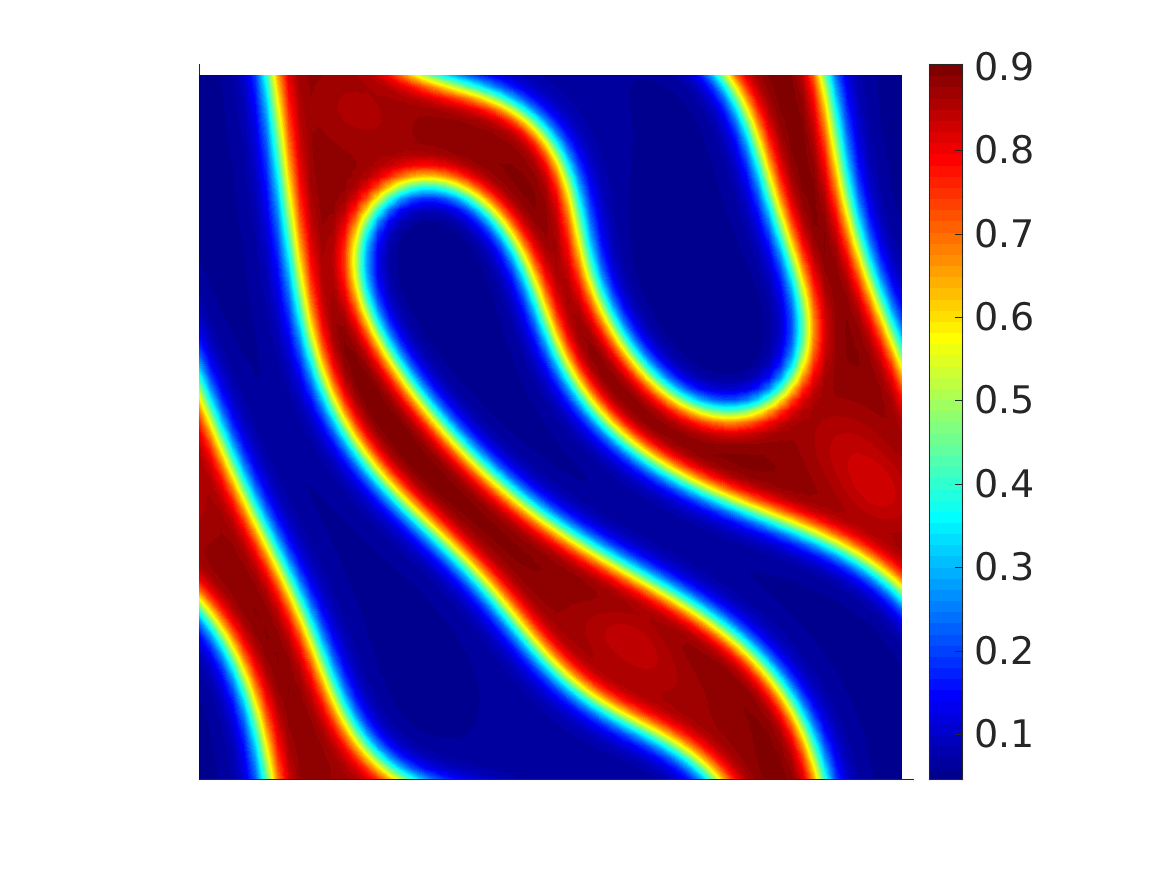} \\[-0.5em]
    \hspace{-1cm}t=3.2 & \hspace{-1cm}t=5.3 & \hspace{-1cm}t=12 \\
\end{tabular}
    \caption{Snapshots of the volume fraction $\phi$ obtained for the second test case, illustrating typical stages of viscoelastic phase separation.\label{fig:evovar}}
\end{figure}
%
In contrast to more standard systems, the phase separation here takes place in several stages~\cite{Tanaka2017}. 
First, the solvent moves out of the polymer forming small droplets which grow over time. 
In the intermediate stage, the polymer starts to form a network-like structure, which finally collapses into the separate phases. 
Due to the small mobility of the polymers, the overall phase separation process is much slower than in symmetric binary fluid systems.
The observed behaviour is typical for the demixing of systems with dynamic asymmetry and is in good agreement with the results presented in~\cite{Tanaka,ZZE}.

\section{Conclusion \& Outlook}
In this work, we proposed and analyzed a fully discrete numerical scheme for a model of viscoelastic phase separation. 
The proposed method is based on variational discretization strategies in space and time, which allows preservation of important structural properties of the system, such as conservation of mass and dissipation of energy, exactly on the discrete level. 
A nonlinear stability analysis for the discrete problem was presented based on relative energies as distance measures. This allowed to establish order optimal convergence rates in space and time under minimal smoothness assumptions, despite the presence of various strong nonlinearities. The discrete stability estimates further allowed to establish uniqueness of the discrete solution under a mild restriction on the time step size. 

The general methodology underlying the proposed numerical method and its analysis can, in principle, be extended to a variety of related nonlinear evolution problems. First results in these directions can be found in \cite{Diss}. The rigorous error analysis in such general cases and further steps towards the efficient solution of the nonlinear systems to be solved in every time step are topics of ongoing research.

{\footnotesize

\section*{Acknowledgement}
A.B. and M.L gratufully acknowledge the support by the German Science Foundation (DFG) via TRR~146 (project~C3) and by the Mainz Institute for Multiscale Modelling. M.L.\ is grateful to the Gutenberg Research College, University Mainz for supporting her research.}

\appendix

\section*{Appendix}

\setcounter{equation}{0}
 \renewcommand{\thesection}{A}
 \renewcommand{\theequation}{\thesection.\arabic{equation}}

For completeness of the presentation, we now provide detailed proofs for some of the technical results that were used in the error analysis of the previous sections.
The appendix is divided into two sections. In Appendix A, we present the projection errors anticipated in the forthcoming estimates. These encompass familiar linear and nonlinear projection errors, along with specific estimates tailored to the errors arising in the analysis.

Appendix B is dedicated primarily to proving Lemma \ref{lem:est_res_full} and Lemma \ref{lem:res_unique}. Within this section, our approach focuses on appropriately estimating the residuals through the relative energy, dissipation, and projection errors. This part is notably technical due to the numerous nonlinearities involved.

\section{Projection error estimates} 
\label{app:proj}

In the following, we summarize some well-known results about standard projection and interpolation operators, which are used in our analysis. 

\subsection{Space discretization}
We consider the setting of Sections~\ref{sec:prelim} and \ref{sec:main} and, in particular, assume (A6)--(A7) to hold true.
The following results then follow with standard arguments; see e.g. \cite{BrennerScott}.
The $L^2$-orthogonal projection $\pi_h^0 : L^2(\Omega) \to \Vh$,
satisfies
\begin{align} \label{eq:l2projest}
    \|u - \pi_h^0 u\|_{H^s} \leq C h^{r-s} \|u\|_{H^r} \qquad \forall u \in H^r(\Omega),
\end{align}
and all parameters $-1 \le s \le r$ and $0 \le r \le 4$. 
On quasi-uniform meshes $\Th$, which we consider here, the projection $\pi_h^0$ is also stable with respect to the $H^1$-norm, i.e., 
\begin{align} \label{eq:h1stab_l2proj}
\|\pi_h^0 u\|_{H^1} \le C \|u\|_{H^1} \qquad \forall u \in H^1(\Omega).
\end{align}
The $H^1$-elliptic projection $\pi_h^1 : H^1(\Omega) \to \Vh$, defined in  \eqref{eq:defh1proj}, satisfies
\begin{align}\label{eq:h1porjest}
 \|u - \pi_h^1 u\|_{H^s} \leq C h^{r-s} \|u\|_{H^r} \qquad \forall u \in H^r(\Omega),
\end{align}
for all parameters $-1 \le s \le r$ and $1 \le r \le 3$. 
Since we assumed quasi-uniformity of the mesh $\Th$, we can further resort to the inverse inequalities
\begin{align} \label{eq:inverse}
    \|v_h\|_{H^1} \le c_{inv} h^{-1} \|v_h\|_{L^2} 
    \qquad \text{and} \qquad 
    \|v_h\|_{L^p} \le c_{inv} h^{d/p-d/q} \|v_h\|_{L^q} 
\end{align}
which hold for all discrete functions $v_h \in \Vh$ and all $1 \le q \le p \le \infty$. 

\subsection{Discrete interpolation}
Let us introduce the discrete Laplacian $\Delta_h:\Vh\to\Vh$ given by
\begin{align}
    \la \Delta_h v_h,w_h \ra = - \la \na v_h,\na w_h \ra, \qquad \forall w_h\in\Vh. \label{eq:disclap}
\end{align}
In particular since $\Delta_h v_h\in \Vh$ the $L^2$-norm can be deduced by setting $w_h=\Delta_h v_h$, i.e,
\begin{equation*}
    \norm{\Delta_h v_h}_0^2 = - \la \na v_h,\na\Delta_h v_h \ra.
\end{equation*}
For a quasi-uniform triangulation, see assumption~(A6), one can obtain
\begin{align}
\norm{\na v_h}_{0,3} &\leq C\norm{\Delta_h v_h}^{1/2}_0\norm{\na v_h}_0^{1/2} + C\norm{\na v_h}_0.  \label{eq:disclapinterpgrad}
\end{align}
A proof of these discrete interpolation inequalities can be found in~\cite{Diegel,Liu2016}.

\subsection{Time discretization}
The piecewise linear interpolation 
$\I_\tau^1:H^1(0,T)\to P_1^c(\Itau)$ and the piecewise constant projection $\bar \pi_\tau^0 : L^2(0,T) \to P_0(\I_\tau)$ in time satisfy
\begin{align} 
    \|u - \bar\pi_\tau^0 u\|_{L^p(0,T)} &\le C \tau^{1/p-1/q+r} \|u\|_{W^{r,q}(0,T)} \quad && \forall u \in W^{r,q}(0,T), \label{eq:timprojest}\\
    \|u - \I_\tau^1 u\|_{L^p(0,T)} &\leq C\tau^{1/p-1/q+2} \|u\|_{W^{r,q}(0,T)} \quad && \forall u \in W^{s,q}(0,T) \label{eq:timinterpest}
\end{align}
with $1 \le p \le q  \le \infty$ and for $0 \le r \le 1$, respectively, $1 \le s \le 2$; see again \cite{BrennerScott}.
Moreover, these operator commute with differentiation in the sense that
\begin{align} \label{eq:commuting} 
\dt (\I_\tau^1 u) = \bar \pi_\tau^0 (\dt u).
\end{align}
We can now further establish the following nonlinear projection error estimates~\cite{brunk2021ch}.

\subsection{Projection estimates for nonlinear terms}

\begin{lemma} \label{lem:average_time_err}
Let $\bar a = \pi_0 a$ denote the $L^2$-orthogonal projection onto $P_{0}(\Itau)$.  
Then for any $u,v \in W^{2,p}(0,T)$ with $1 \le p \le \infty$, one has 
\begin{align}
\|\overline{\bar u\bar v}-\overline{uv}\|_{L^p(0,T)} 
&\le C \tau^{2} \|u\|_{W^{2,p}(0,T)} \|v\|_{W^{2,p}(0,T)},  \label{eq:midpoint_order}
\end{align}
with a constant $C$ independent of $u$ and $v$. 
\end{lemma}
 In a similar manner, we obtain the following estimate~\cite{brunk2021ch}.
\begin{lemma}
Let $\bar a = \pi_0 a$ denote the $L^2$-orthogonal projection onto $P_{0}(\Itau)$. Furthermore, let $\phi\in P_1(\Itau)$. 
Then for any $u,v \in W^{2,p}(0,T)$ with $1 \le p \le \infty$, one has 
\begin{align}
\| g(\bar\phi)-\overline{g(\phi)}\|_{L^p(0,T)} 
&\le C \tau^{2} \|g(\phi)\|_{W^{2,p}(0,T)},  \label{eq:midpoint_order_single}
\end{align}
with a constant $C$ depending only on the polynomial degree $k$. 
\end{lemma}

The following nonlinear projection estimates will be required for estimating the residual terms in the following section.
\begin{lemma} \label{lem_prodest}
Let $(\hat\phi_{h,\tau},\hbmu_{h,\tau},\hat q_{h,\tau})$ given by \eqref{eq:fullproj_1}. Then the following estimates hold
\begin{align*}      \int_{I_n}&\norm{A'(\hbphi_{h,\tau})^2\hbq_{h,\tau}\snorm{\na\hbphi_{h,\tau}}^2 - \overline{A'(\phi))^2 q\snorm{\na\phi}^2}}_{0,6/5}^2 \leq C_1h^4 + C_2\tau^4,\\
\int_{I_n}& \norm{(A\cdot A')(\hbphi_{h,\tau})\na\hbq_{h,\tau}\na\hbphi_{h,\tau}  - \overline{(A\cdot A')(\phi)\na q\na\phi}}_{0,6/5}^2\leq C_3h^4 + C_4\tau^4, \\
\int_{I_n}& \norm{A'(\hbphi_{h,\tau})b^{1/2}(\hbphi_{h,\tau})\na\hbmu_{h,\tau}\na\hbphi_{h,\tau} - \overline{A'(\phi)b^{1/2}(\phi)\na\mu\na\phi}}_{0,6/5}^2 \leq C_5h^4 + C_6\tau^4. 
\end{align*}
\end{lemma}
\begin{proof}
We introduce the abbreviations $B_{1}(\cdot)=A'(\cdot)^2, B_{2}(\cdot)=A'(\cdot)A(\cdot) $
We consider the first term and by addition of suitable zeros we estimate
\begin{align*}
&\int_{I_n}\norm{B_1(\hbphi_{h,\tau})\hbq_{h,\tau}\snorm{\na\hbphi_{h,\tau}}^2 - \overline{B_1(\phi)) q\snorm{\na\phi}^2}}_{0,6/5}^2 \\
&\leq \int_{I_n}\norm{(B_1(\hbphi_{h,\tau})-\overline{B_1(\phi)})\bar q \snorm{\nabla\bar\phi}^2}_{0,6/5}^2
+ \norm{B_1(\hbphi_{h,\tau})\snorm{\nabla\bar\phi}^2(\hbq_{h,\tau}-\bar q)}_{0,6/5}^2 \\
&\, + \norm{B_1(\hbphi_{h,\tau})\hbq_{h,\tau}(\nabla\hbphi_{h,\tau}+\bar\phi)(\nabla\hbphi_{h,\tau}-\bar\phi)}_{0,6/5}^2
+ \norm{\overline{B_1(\phi)}\bar q\snorm{\nabla\bar\phi}^2 - \overline{B_1(\phi) q\snorm{\nabla\phi}^2}}_{0,6/5}^2 \\
& = (a) + (b) + (c) + (d).
\end{align*}
For the first term we use Lemma \ref{lem:average_time_err} and estimate
\begin{align*}
  (a) & \leq  \norm{\bar q}_{0,\infty}^2\norm{\nabla\bar\phi}_{0,3}^4\int_{I_n}C\norm{\hbphi_{h,\tau} - \bar\phi}_{0,6}^2 + \int_{I_n}\norm{B_1(\bar\phi)-\overline{B_1(\phi)}}_{0,6}^2 \\
  & \leq \norm{ q}_{L^\infty(L^\infty)}^2\norm{\nabla\phi}_{L^\infty(L^3)}^4(\tau^4\norm{B_1(\phi)}^2_{H^2(L^6)} + h^4\norm{\phi}_{L^2(H^3)}^2).
\end{align*}
The second term can be bounded by
\begin{align*}
 (b) & \leq C\norm{\nabla\bar\phi}_{0,3}^4\int_{I_n} \norm{\hbq_{h,\tau}-\bar q}^2_{0,6}  \leq  Ch^4\norm{q}_{L^2(H^3)}^2\norm{\nabla\phi}_{L^\infty(L^3)}^4.
\end{align*}
For the third term, we can use 
\begin{align*}
 (c) & \leq  \norm{\nabla\bar\phi}_{0,3}^2 \norm{\hbq_{h,\tau}}^2_{0,\infty}  \int_{I_n}\norm{\nabla\hbphi_{h,\tau}-\bar\phi}_{0,2}^2 
 \leq Ch^4\norm{\nabla\phi}_{L^\infty(L^3)}^2 \norm{q}^2_{L^\infty(L^\infty)}\norm{\phi}_{L^2(H^3)}.  
\end{align*}
The last term can again be treated by Lemma~\ref{lem:average_time_err} which leads to
\begin{align*}
 (d) \leq \tau^4\norm{B_1(\phi)q\snorm{\nabla\phi}^2}_{H^2(L^{6/5})}^2.   
\end{align*}
In summary, this yields the first bound of the lemma.
The second bound stated in the lemma can be rewritten and estimated, using similar arguments, by
\begin{align*}
\int_{I_n}\norm{B_2(\hbphi_{h,\tau})&\nabla\hbq_{h,\tau}\na\hbphi_{h,\tau} - \overline{B_2(\phi)) \nabla q\na\phi}}_{0,6/5}^2 \\
&\leq \int_{I_n}\norm{(B_2(\hbphi_{h,\tau})-\overline{B_2(\phi)})\nabla\bar q \nabla\bar\phi}_{0,6/5}^2 
+ \norm{B_2(\hbphi_{h,\tau})\nabla\bar\phi\nabla(\hbq_{h,\tau}-\bar q)}_{0,6/5}^2 \\
& \,
+ \norm{B_2(\hbphi_{h,\tau})\nabla\hbq_{h,\tau}(\nabla\hbphi_{h,\tau}-\bar\phi)}_{0,6/5}^2
+ \norm{\overline{B_2(\phi)}\nabla\bar q\nabla\bar\phi - \overline{B_2(\phi) \nabla q\nabla\phi}}_{0,6/5}^2.
\end{align*}
The individual terms can now be estimated as before.
The integral in the last bound of the lemma can be treated like the second one after replacing $B_2$ by $A'(\cdot)b^{1/2}(\cdot)$ and $q$ by $\mu$.
\end{proof}

\renewcommand{\thesection}{B}
\section{Proof of Lemma~\ref{lem:est_res_full}}

We have now assembled all ingredients for the proof of Lemma~\ref{lem:est_res_full}. Without further mentioning, we assume the conditions of Lemma~\ref{lem:est_res_full} to be valid. We estimate the three residuals separately.

\subsection*{First residual} 
By definition of the dual norm~\eqref{eq:dualnorm} and splitting the terms,  we get 
\begin{align*}
 \int_{I_n} \norm{\bar  r_{1,h,\tau}}_{-1,h}^2 \ddta &\leq  3 \int_{I_n} \norm{\dt(\pi^1_h\phi-\phi)}_{-1,h}^2 + \norm{b(\bar\phi_{h,\tau})\na\bmu_{h,\tau} - \overline{b(\phi)\na\mu}}_0^2 \\
 &\qquad \qquad + \norm{c(\bar\phi_{h,\tau})\na(A(\bphi_{h,\tau})\hbq_{h,\tau}) - \overline{c(\phi)\na(A(\phi)q)}}_0^2 
  \ddta \\
& = (i) + (ii) + (iii) .
\end{align*}
The first term can be bounded using interpolation error estimates by
\begin{equation*}
(i) \leq Ch^4\norm{\dt\phi}_{L^2(H^1)}^2.    
\end{equation*}
The second term can be further split into multiple parts according to
\begin{align*}
(ii) & \leq C\int_{I_n}\norm{b(\bar\phi_{h,\tau})\na(\hbmu_{h,\tau} - \na\bar\mu)}_0^2 + \norm{(b(\bar\phi_{h,\tau}) - b(\bar\phi))\na\bar\mu}_0^2 \\
& \qquad \qquad  + \norm{(b(\bar\phi) - \overline{b(\phi))\na\bar\mu}}_0^2 + \norm{\overline{b(\phi)}\na\bar\mu - \overline{b(\phi)\na\mu}}_0^2  \ddta   \\
 &= (a) + (b) + (c) + (d).
\end{align*}
By the stability of the $L^2$-projection in time and the estimates of Lemma~\ref{lem:projerr}, we get
\begin{equation*}
 (a) \leq C\int_{I_n}\norm{\pi_h^0\mu-\mu}_1^2  \ddta \leq Ch^4\norm{\mu}_{L^2(H^3)}^2. 
\end{equation*}
With the bounds for the derivatives of the parameter function $b(\cdot)$, the Hölder inequality, the stability of the $L^2$-projection, and Lemma~\ref{lem:projerr}, we further get
\begin{align*}
(b) &\leq C\int_{I_n}\norm{\phi_{h,\tau}-\phi}_{0,6}^2\norm{\mu}_{1,3}^2 \ddta 
\leq Ch^4\norm{\mu}_{L^\infty(W^{1,3})}^2\norm{\phi}_{L^2(H^3)}^2 \\
& \qquad \qquad + C\tau^4\norm{\mu}_{L^\infty(W^{1,3})}^2\norm{\phi}_{H^2(H^1)}^2 + \norm{\mu}_{L^\infty(W^{1,3})}^2\int_{I_n} \E^\phi_\alpha(\phi_{h,\tau}|\hat\phi_{h,\tau}) \ddta. 
\end{align*}
The third term can be treated by Lemma~\ref{eq:midpoint_order_single} and yields
\begin{equation*}
(c) \leq C\tau^4\norm{\mu}_{L^\infty(W^{1,3})}^2\norm{b(\phi)}_{H^2(L^6)}^2.  
\end{equation*}
The last term in the above expansion can be treated by Lemma~\ref{eq:midpoint_order} which leads to 
\begin{equation*}
(d) \leq C\tau^4\norm{b(\phi)\na\mu}_{H^2(L^2)}^2.   
\end{equation*}
The last norm can be further expanded using the product and chain rule of differentiation.
All terms arising in these computations can be controlled appropriately.
In summary, we thus get
\begin{align*}
 (ii) \leq C_1(\phi,\mu)h^4 + C_2(\phi,\mu)\tau^4 + C(\norm{\mu}_{L^\infty(W^{1,3})})\int_{I_n} \E_\alpha(\phi_{h,\tau}|\hat\phi_{h,\tau}) \ddta.  
\end{align*}
The constants depend on 
norms of the solution that are bounded by our assumptions.
We continue with the third term, which can be estimated by
\begin{align*}
(iii) &\leq  \int_{I_n} \norm*{c(\bphi_{h,\tau})A(\bphi_{h,\tau})\na\hbq_{h,\tau} - \overline{c(\phi)A(\bphi_{h,\tau})\nabla q} }^2_0 \\
&\qquad \qquad + \norm*{c(\bphi_{h,\tau})\hbq_{h,\tau}A^{'}(\bphi_{h,\tau})\na\bphi_{h,\tau} - \overline{c(\phi)qA^{'}(\phi)\nabla \phi} }^2_0 \ddta. 
\end{align*}
The first part can be estimated similar to term $(i)$ before, which leads to
\begin{align*}
(iiia) \leq C_1(\phi,q)h^4 + C_2(\phi,q)\tau^4 + C(\norm{q}_{L^\infty(W^{1,3})})\int_{I_n} \E^\phi_\alpha(\phi_{h,\tau}|\hat\phi_{h,\tau}) \ddta.
\end{align*}
The two constants 
again depend on the norms of the solution that are bounded by assumption.
The second part can be further split and estimated by 
\begin{align*}
(iiib) 
&\leq 2 \int_{I_n}\norm*{c(\bphi_{h,\tau})A^{'}(\bphi_{h,\tau})\na\bphi_{h,\tau} - \overline{c(\phi)A^{'}(\phi)\nabla \phi} }^2_0\norm{\hbq_{h,\tau}}_{0,\infty}^2 \\
&\qquad \qquad + \norm*{\overline{c(\phi)\hbq_{h,\tau}A^{'}(\phi)\nabla \phi}- \overline{c(\phi) q A^{'}(\phi)\nabla \phi} }^2_0 \ddta  \\
&\le \int_{I_n} C\E_\alpha^\phi(\phi_{h,\tau}|\hat\phi_{h,\tau}) + C_3(\phi,q)h^4 + C_4(\phi,q)\tau^4,    
\end{align*}
where we have used similar arguments as in the previous steps. 
The constants again only depend of bounds for the parameters and the solutions that are available from our assumptions. 
By combination of all estimates, we can finally bound the first residual by
\begin{equation*}
\int_{I_n} \norm{\bar  r_{1,h,\tau}}_{-1,h}^2 \ddta \leq C(h^4 + \tau^4) + \int_{I_n}C \E^\phi_\alpha(\phi_{h,\tau}|\hat\phi_{h,\tau}) \ddta.  
\end{equation*}
The constants in this estimate are independent of the discretization parameters.

\subsection*{Second residual} 
The second residual can be expressed in the strong form as 
\begin{align*}
    \bar r_{2,h,\tau} = (\overline{\pi_h^0 \mu} - \overline{\I_\tau^1 \pi_h^0 \mu}) + (\overline{\I_\tau^1 \phi} - \overline{\hat \phi_{h,\tau}}) + (\overline{f'(\hat \phi_{h,\tau})} - \overline{\I_\tau^1 f'(\phi)}).
\end{align*}
Recall that $\overline{g} = \bar \pi_\tau^0 g$ is used to denote the piecewise constant projection of a function $g$ with respect to time. 
This pointwise representation allows us to estimate
\begin{align*}
\frac{1}{3} \int_{I_n} \|\bar r_{2,h,\tau}\|_1^2 \ddta 
&\leq \|\pi_h^0 \mu - \I_\tau^1 \pi_h^0 \mu\|^2_{L^2(H^1_p)} + \|\I_\tau^1 \phi - \hat \phi_{h,\tau}\|^2_{L^2(H^1_p)} \\
&+ \|f'(\hat \phi_{h,\tau}) - \I_\tau^1 f'(\phi)\|^2_{L^2(H^1_p)}  \\
&= (i) + (ii) + (iii) .   
\end{align*}
Using the contraction property of the $L^2$-projection in space, we obtain for the first term
\begin{equation*}
(i) \leq \norm{\mu - I_\tau^1\mu}_{L^2(H^1)}^2 \leq C\tau^4\norm{\mu}_{H^2(H^1)}^2.
\end{equation*}
With the error estimate for the $H^1$-projection $\pi_h^1$, we further find 
\begin{equation*}
(ii) \leq C\norm{\phi-\pi_h^1\phi}^2_{L^\infty(H^1)} \leq Ch^4\norm{\phi}^2_{L^\infty(H^3)}.    
\end{equation*}
For the last term we employ the uniform bounds of $\phi$ and $\hat\phi_{h,\tau}$ in $L^\infty(0,T;W^{1,\infty}(\Omega))$. Hence all terms of the form $f^{(k)}(\cdot)$ can be bounded uniformly by a constant $C(f)$, and we obtain
\begin{align*}
(iii) &\leq 2 \norm{f'(\hat\phi_{h,\tau})-f'(\phi)}_{L^2(H^1)}^2 + 2 \norm{f'(\phi) - I^1_\tau f'(\phi)}_{L^2(H^1)}^2  \\
&\leq C(f)\norm{\hat\phi_{h,\tau}-\phi}_{L^2(H^2)}^2 + C\tau^4\norm{f'(\phi)}^2_{H^2(H^1)} \\
&\leq C'(f)h^4\norm{\phi}_{L^2(H^3)}^2 + C'(f)\tau^4\norm{\phi}_{H^2(H^1)}^2 + C\tau^4\norm{f'(\phi)}^2_{H^2(H^1)}.
\end{align*}
The terms involving derivatives of $f(\phi)$ can all be estimated due to the regularity assumptions on $f$ and $\phi$. In summary, we obtain the following bound for the second residual
\begin{equation*}
 \int_{I_n}\norm{\bar r_{2,h,\tau}}_{1}^2 \ddta \leq C(h^4 + \tau^4).  
\end{equation*}
The constant is again independent of the discretization parameters. 

\subsection*{Third residual}
Due to many nonlinearities, this is the most technical part of our estimates. As a preliminary step, we decompose
\begin{align*}
& d_0 \la \na(A(\bphi_{h,\tau})\hbq_{h,\tau}),\na(A(\bphi_{h,\tau})\bar\zeta_{h,\tau}) \ra \\
&= d_0 \la A^2(\bphi_{h,\tau})\na\hbq_{h,\tau},\na\bar\zeta_{h,\tau} \ra
+ d_0 \la (A'(\bphi_{h,\tau}))^2\hbq_{h,\tau}\snorm{\na\bphi_{h,\tau}}^2,\bar\zeta_{h,\tau} \ra \\
&\qquad \qquad + d_0 \la (A\cdot A')(\bphi_{h,\tau})\hbq_{h,\tau}\na\bphi_{h,\tau},\na\bar\zeta_{h,\tau} \ra  + d_0 \la (A\cdot A')(\bphi_{h,\tau})\na\hbq_{h,\tau}\na\bphi_{h,\tau}, \bar\zeta_{h,\tau} \ra,  
\end{align*}
and in a similar manner, we also split
\begin{align*}
&\la c(\bphi_{h,\tau})\na\hbmu_{h,\tau},\na(A(\bphi_{h,\tau})\bar\zeta_{h,\tau}) \\
&= \la A(\bphi_{h,\tau})c(\bphi_{h,\tau})\na\hbmu_{h,\tau},\na\bar\zeta_{h,\tau} \ra + \la A'(\bphi_{h,\tau}) c(\bphi_{h,\tau})\na\hbmu_{h,\tau}\na\bphi_{h,\tau},\bar\zeta_{h,\tau} \ra.   
\end{align*}
From the definition of the dual norm~\eqref{eq:dualnorm}, the binomial inequality, and the bound for the coefficient $d_0$, we then obtain
\begin{align*}
 &\int_{I_n}\norm{\bar r_{3,h,\tau}}_{-1,h}^2 \ddta \leq  8 \int_{I_n} \norm{\kappa^{1/2}(\bphi_{h,\tau})\hbq_{h,\tau}-\overline{\kappa^{1/2}(\phi) q}}_0^2 \\
 &\qquad \qquad + \norm{A^2(\bphi_{h,\tau})\na\hbq_{h,\tau} - \overline{A^2(\phi)\na q}}_0^2 \\
 &\qquad \qquad+ \norm{A'(\bphi_{h,\tau}))^2\hbq_{h,\tau}(\na\bphi_{h,\tau})^2 - \overline{A'(\phi))^2 q (\na\phi)^2}}^2_{0,6/5}  \\
 &\qquad \qquad +\norm{(A\cdot A')(\bphi_{h,\tau})\hbq_{h,\tau}\na\bphi_{h,\tau} - \overline{(A\cdot A')(\phi) q\na\phi}}_0^2 \\
 &\qquad \qquad + \norm{(A\cdot A')(\bphi_{h,\tau})\na\hbq_{h,\tau}\na\bphi_{h,\tau} - \overline{(A\cdot A')(\phi)\na q\na\phi}}_{0,6/5}^2 \\
 &\qquad \qquad  + \norm{A(\bphi_{h,\tau}) c(\bphi_{h,\tau})\na\hbmu_{h,\tau} - \overline{A(\phi) c(\phi)\na\mu}}_0^2 \\
 &\qquad \qquad + \norm{A'(\bphi_{h,\tau}) c(\bphi_{h,\tau})\na\hbmu_{h,\tau}\na\bphi_{h,\tau} - \overline{A'(\phi) c(\phi)\na\mu\na\phi}}_{0,6/5}^2 + \norm{\nabla(\hbq_{h,\tau}-\bar q)}_0^2\ddta \\
 & = (i) + (ii) + (iii) + (iv) + (v) + (vi) + (vii) + (viii).
\end{align*}
With similar arguments as used for bounding the first residual above, we obtain
\begin{align*}
 (i)\leq C_1 h^4 + C_2\tau^4 + C(\norm{q}_{L^\infty({W^{1,3})}},\norm{\mu}_{L^\infty(W^{1,3})})\int_{I_n} \E^\phi_\alpha(\phi_{h,\tau}|\hat\phi_{h,\tau}) \ddta. 
\end{align*}
The constants $C_1$, $C_2$ again only depend on quantities that can be controlled by our assumptions. 
The terms $(ii)$, $(iv)$, and $(vi)$ can be estimated in the same manner as the terms $(ii)$ and $(iii)$ in the first residual and term $(viii)$ is bounded by the projection error Lemma \ref{lem:projerr}; the details are therefore omitted.
For the first term containing the $L^{6/5}$-norm, we have
\begin{align*}
(iii) &\leq  2 \int_{I_n} \norm{A'(\bphi_{h,\tau})^2(\snorm{\na\bphi_{h,\tau}}^2 -\snorm{\na\hbphi_{h,\tau}}^2)}_{0,6/5}^2\norm{\hbq_{h,\tau}}^2_{0,\infty} \\
&\qquad \qquad + \norm{A'(\bphi_{h,\tau})^2\hbq_{h,\tau}\snorm{\na\hbphi_{h,\tau}}^2 - \overline{A'(\phi))^2 q\snorm{\na\phi}^2}}_{0,6/5}^2  \ddta 
= (a) + (b).
\end{align*}
With the bounds for $A'$ and the discrete interpolation inequality \eqref{eq:disclapinterpgrad}, we obtain
\begin{align*}
(a) &\leq C\int_{I_n}\norm{\na\bphi_{h,\tau} -\na\hbphi_{h,\tau}}_{0,3}^2\norm{\na\bphi_{h,\tau} +\na\hbphi_{h,\tau}}_{0,2}^2\norm{\hbq_{h,\tau}}^2_{0,\infty}  \\
&\leq C\int_{I_n}\big(\norm{\na\bphi_{h,\tau} -\na\hbphi_{h,\tau}}_{0,2}\norm{\Delta_h(\bphi_{h,\tau} -\hbphi_{h,\tau})}_{0,2} + \E^\phi_\alpha(\phi_{h,\tau}|\hat \phi_{h,\tau})\big) \norm{\hbq_{h,\tau}}^2_{0,\infty} \\
& \leq C(\delta)\int_{I_n} \E^\phi_\alpha(\phi_{h,\tau}|\hat \phi_{h,\tau})\norm{\hbq_{h,\tau}}^4_{0,\infty} \ddta + \delta\int_{I_n}\norm{\Delta_h(\bphi_{h,\tau} -\hbphi_{h,\tau})}_{0,2}^2. 
\end{align*}
The second term in the first line is controlled by the uniform bounds for $\phi_{h,\tau}$, $\hat \phi_{h,\tau}$ in $L^\infty(H^1)$.
The parameter $\delta>0$ will be chosen later at our convenience. 
Using \eqref{eq:pg2}, $\eqref{eq:pgp2}$, and the definition of the discrete Laplacian, we further find
\begin{align}
  \int_{I_n}\norm{\Delta_h(\bphi_{h,\tau} &-\hbphi_{h,\tau})}_{0,2}^2 \leq \int_{I_n}\norm{\overline{f'(\phi_{h,\tau})-f'(\hat\phi_{h,\tau})}}_0^2 + \norm{\bmu_{h,\tau}-\hbmu_{h,\tau}}_0^2 + \norm{\bar r_{2,h,\tau}}_0^2 \ddta \notag \\
  &\leq \int_{I_n}C(f)\E^\phi_\alpha(\phi_{h,\tau}|\hat \phi_{h,\tau}) + \delta\D_{\bar\phi_{h,\tau}}(\bmu_{h,\tau}-\hbmu_{h,\tau}) + C\norm{\bar r_{2,h,\tau}}_1^2 \ddta. \label{eq:disclaplchemest}
\end{align}
In summary, this leads to the bound for $(a)$. 
For the second term, we use Lemma~\ref{lem_prodest} to see that
\begin{align*}
 (b) & \leq  \int_{I_n}\norm{[A'(\bphi_{h,\tau})^2-A'(\hbphi_{h,\tau})^2]\hbq_{h,\tau}\snorm{\na\hbphi_{h,\tau}}^2}_{0,6/5}^2 \\
 & \qquad \qquad + \norm{A'(\hbphi_{h,\tau})^2\hbq_{h,\tau}\snorm{\na\hbphi_{h,\tau}}^2 - \overline{A'(\phi))^2 q\snorm{\na\phi}^2}}_{0,6/5}^2  \\
 & \leq \int_{I_n}\norm{\hbq_{h,\tau}}^2_{0,\infty}\norm{\nabla\hbphi_{h,\tau}}_{0,3}^4\E^\phi_\alpha(\phi_{h,\tau}|\hat \phi_{h,\tau}) + C_1h^4 + C_2\tau^4.
\end{align*}
For the sixth term in the above error decomposition, we again use Lemma~\ref{lem_prodest} to get
\begin{align*}
 (v) &\leq \int_{I_n}   \norm{(A\cdot A')(\bphi_{h,\tau})\na\hbq_{h,\tau}\na\bphi_{h,\tau} - (A\cdot A')(\bphi_{h,\tau})\na\hbq_{h,\tau}\na\hbphi_{h,\tau}}_{0,6/5}^2  \\
 &\qquad \qquad +  \norm{(A\cdot A')(\bphi_{h,\tau})\na\hbq_{h,\tau}\na\hbphi_{h,\tau}  - \overline{(A\cdot A')(\phi)\na q\na\phi}}_{0,6/5}^2 \\
 & \leq C\norm{\nabla \hbq_{h,\tau}}^2_{0,3}\norm{\nabla(\bphi_{h,\tau}-\hbphi_{h,\tau})}^2_{0,2} + C\norm{\bphi_{h,\tau}-\hbphi_{h,\tau}}_{0,6}^2\norm{\na\hbq_{h,\tau}\na\hbphi_{h,\tau}}_{0,3/2}^2 \\
 &\qquad \qquad + \norm{(A\cdot A')(\hbphi_{h,\tau})\na\hbq_{h,\tau}\na\hbphi_{h,\tau}  - \overline{(A\cdot A')(\phi)\na q\na\phi}}_{0,6/5}^2 \\
 &\leq C(\norm{\nabla \hbq_{h,\tau}}^2_{0,3} + \norm{\na\hbq_{h,\tau}\na\hbphi_{h,\tau}}_{0,3/2}^2)\E^\phi_\alpha(\phi_{h,\tau}|\hat \phi_{h,\tau})  + C_3h^4 + C_4\tau^4.
\end{align*}
For the remaining term, we obtain in a similar manner
\begin{align*}
 (vii) &\leq \int_{I_n}   \norm{(A\cdot A')(\bphi_{h,\tau})\na\hbmu_{h,\tau}\na\bphi_{h,\tau} - (A\cdot A')(\bphi_{h,\tau})\na\hbmu_{h,\tau}\na\hbphi_{h,\tau}}_{0,6/5}^2  \\
 &\qquad \qquad +  \norm{(A\cdot A')(\bphi_{h,\tau})\na\hbmu_{h,\tau}\na\hbphi_{h,\tau}  - \overline{(A\cdot A')(\phi)\na\mu\na\phi}}_{0,6/5}^2 \\
 & \leq C\norm{\nabla \hbmu_{h,\tau}}^2_{0,3}\norm{\nabla(\bphi_{h,\tau}-\hbphi_{h,\tau})}^2_{0,2} + C\norm{\bphi_{h,\tau}-\hbphi_{h,\tau}}_{0,6}^2\norm{\na\hbmu_{h,\tau}\na\hbphi_{h,\tau}}_{0,3/2}^2 \\
 &\qquad \qquad + \norm{(A\cdot A')(\hbphi_{h,\tau})\na\hbmu_{h,\tau}\na\hbphi_{h,\tau}  - \overline{(A\cdot A')(\phi)\na\mu\na\phi}}_{0,6/5}^2 \\
 &\leq C(\norm{\nabla \hbmu_{h,\tau}}^2_{0,3} + \norm{\na\hbmu_{h,\tau}\na\hbphi_{h,\tau}}_{0,3/2}^2)\E^\phi_\alpha(\phi_{h,\tau}|\hat \phi_{h,\tau})  + C_3h^4 + C_4\tau^4.
\end{align*}
In total the third residual can therefore be estimated by
\begin{equation*}
 \int_{I_n}\norm{\bar r_{3,h,\tau}}_{-1,h}^2 \ddta \leq C(h^4 + \tau^4) + \int_{I_n} C\E_\alpha(\phi_{h,\tau},q_{h,\tau}|\hat \phi_{h,\tau},\hat q_{h,\tau}) \ddta.  
\end{equation*}

\section{Proof of Lemma~\ref{lem:unique_res}}

We assume the conditions of Lemma~\ref{lem:unique_res} to hold and again estimate the two residuals separately. 

\subsection*{First residual.}
With similar arguments as used in the previous section, we obtain 
\begin{align*}
\int_{I_n} \norm{\bar r_{1,h,\tau}}_{-1,h}^2 \ddta 
& \leq C \int_{I_n}\norm{b'}_{0,\infty}^2\norm{\hat\mu}_{1,3}^2\norm{\hat\phi_{h,\tau}-\hbphi_{h,\tau}}_{0,6}^2 \\
& \qquad \qquad + \norm{(c A)'}^2_{0,\infty}\norm{\hbq_{h,\tau}}_{1,3}\norm{\hat\phi_{h,\tau}-\hbphi_{h,\tau}}_{0,6}^2 \\
& \qquad \qquad + \norm{c A'}^2_{0,\infty}\norm{\hbq_{h,\tau}}_{0,\infty}^2\norm{\hat\phi_{h,\tau}-\hbphi_{h,\tau}}_1^2 \\
& \qquad \qquad + \norm{(cA')'}^2_{0,\infty}\norm{\hbq_{h,\tau}}_{0,\infty}^2\norm{\hat\phi_{h,\tau}-\hbphi_{h,\tau}}_{0,6}^2\norm{\hbphi_{h,\tau}}_{1,3}^2\\
&\leq 
C_1 \int_{I_n}\E^\phi_\alpha(\phi_{h,\tau}|\hat\phi_{h,\tau}) \ddta
\end{align*}
with constant $C_1$ depending on $\norm{\hbmu_{h,\tau}}_{L^\infty(W^{1,3})}$, $\norm{\hbq_{h,\tau}}_{L^\infty(W^{1,3})}$, and $\norm{\hbphi_{h,\tau}}_{L^\infty(W^{1,3})}$.
The third residual can be estimated similarly, which leads to
\begin{align*}
&\int_{I_n} \norm{\bar r_{3,h,\tau}}_{-1,h}^2 \ddta 
\leq C \int_{I_n} \norm{\kappa'}_{0,\infty}^2\norm{\bar q_{h,\tau}-\hbq_{h,\tau}}_0^2 
+ \norm{cA'}^2_{0,\infty}\norm{\hbmu_{h,\tau}}^2_{1,3}\norm{\bar\phi_{h,\tau}-\hbphi_{h,\tau}}^2_{0,6}\\
& \qquad + \norm{cA'}^2_{0,\infty}\norm{\hbmu_{h,\tau}}_{1,3}^2\norm{\na(\bar\phi_{h,\tau}-\hbphi_{h,\tau})}_0^2 
\\
&\qquad+ \norm{cA'}^2_{0,\infty}\norm{\hbmu_{h,\tau}}_{1,3}^2\norm{\bar\phi_{h,\tau}-\hbphi_{h,\tau}}_{0,6}^2\norm{\hbphi_{h,\tau}}^2_{1,3} \\
& \qquad + \norm{(A^2)'}_{0,\infty}^2\norm{\hbq_{h,\tau}}_{1,3}^2\norm{\bar\phi_{h,\tau}-\hbphi_{h,\tau}}_{0,6}^2 \\
&\qquad + \norm{(A')^2}_{0,\infty}\norm{\hbq_{h,\tau}}^2_{0,\infty}\norm{\na(\bar\phi_{h,\tau}+\hbphi_{h,\tau})}_0^2\norm{\na(\bar\phi_{h,\tau}-\hbphi_{h,\tau})}_{0,3}^2\\
& \qquad + \norm{((A')^2)'}_{0,\infty}\norm{\hbq_{h,\tau}}^2_{0,\infty}\norm{\bar\phi_{h,\tau}-\hbphi_{h,\tau}}_{0,6}^2\norm{\na\hbphi_{h,\tau}}_0^2\\
& \qquad + \norm{AA'}_{0,\infty}^2(\norm{\hbq_{h,\tau}}_{0,\infty}^2 + \norm{\hbq_{h,\tau}}_{1,3}^2)\norm{\na(\bar\phi_{h,\tau}-\hbphi_{h,\tau})}_0^2 \\
& \qquad + \norm{(AA')'}_{0,\infty}^2(\norm{\hbq_{h,\tau}}_{0,\infty}^2 + \norm{\hbq_{h,\tau}}_{1,3}^2)\norm{\bar\phi_{h,\tau}-\hbphi_{h,\tau}}_0^2\norm{\na\hbphi_{h,\tau}}_{0,3}^2 \\
& \leq  
C_2 \int_{I_n}\E_\alpha(z_{h,\tau}|\hat z_{h,\tau}) \ddta 
+ 
C_3 \int_{I_n} \norm{\na(\bar\phi_{h,\tau}-\hbphi_{h,\tau})}_{0,3}^2 \ddta.
\end{align*}
The constants $C_2$, $C_3$  depend on $\norm{\hbmu_{h,\tau}}_{L^\infty(W^{1,3})}$, $\norm{\hbq_{h,\tau}}_{L^\infty(W^{1,3})}$, $\norm{\hbq_{h,\tau}}_{L^\infty(L^\infty)}$, $\norm{\hbphi_{h,\tau}}_{L^\infty(W^{1,3})}$
and $\norm{\hbq_{h,\tau}}^2_{L^\infty(L^\infty)}$, $\norm{\na(\bar\phi_{h,\tau}+\hbphi_{h,\tau})}_{L^\infty(L^2)}^2$, respectively.
The last term in the above estimate can again be estimated by the discrete interpolation inequality \eqref{eq:disclapinterpgrad} and estimates for the discrete Laplacian, which finally can be absorbed in the dissipation terms; compare with \eqref{eq:disclaplchemest}.
In summary, we thus have obtained the required estimates for the two residuals of Lemma~\ref{lem:unique_res}.

\bibliographystyle{abbrv}
\bibliography{relenergy}

\begin{thebibliography}{10}

\bibitem{Akrivis19}
G.~Akrivis, B.~Li, and D.~Li.
\newblock Energy-decaying extrapolated {RK}--{SAV} methods for the {A}llen--{C}ahn and {C}ahn--{H}illiard equations.
\newblock {\em SIAM J. Sci. Comput.}, 41:A3703--A3727, 2019.

\bibitem{Sueli07}
J.~W. Barrett and E.~S\"{u}li.
\newblock Existence of global weak solutions to some regularized kinetic models for dilute polymers.
\newblock {\em Multiscale Model Simul}, 6(2):506--546, 2007.

\bibitem{Braukhoff2021}
M.~Braukhoff and A.~Jüngel.
\newblock Entropy-dissipating finite-difference schemes for nonlinear fourth-order parabolic equations.
\newblock {\em Discrete Contin. Dyn. Syst. Ser. B}, 26:3335--3355, 2021.

\bibitem{BrennerScott}
S.~C. Brenner and L.~R. Scott.
\newblock {\em The mathematical theory of finite element methods}, volume~15 of {\em {Texts in applied mathematics}}.
\newblock Springer, New York, 3 edition, 2008.

\bibitem{Diss}
A.~Brunk.
\newblock {\em {Viscoelastic phase separation: Well-posedness and numerical analysis}}.
\newblock PhD thesis, Johannes Gutenberg University Mainz, 2022.
\newblock \url{https://openscience.ub.uni-mainz.de/handle/20.500.12030/6777}.

\bibitem{Brunk.Ex3D}
A.~Brunk.
\newblock Existence and weak-strong uniqueness for global weak solutions for the viscoelastic phase separation model in three space dimensions.
\newblock {\em Discrete Contin. Dyn. Syst.}, pages 0--0, 2023.

\bibitem{brunk2021ch}
A.~Brunk, H.~Egger, O.~Habrich, and M.~Luk\'{a}\voo{c}ov\'{a}-Medvid'ov\'{a}.
\newblock Stability and discretization error analysis for the {C}ahn–{H}illiard system via relative energy estimates.
\newblock {\em ESAIM: M2AN}, 57:1297--1322, 2023.

\bibitem{Brunk.Exreg}
A.~Brunk and M.~Luk\'{a}\voo{c}ov\'{a}-Medvid'ov\'{a}.
\newblock Global existence of weak solutions to viscoelastic phase separation {p}art: I. {R}egular case.
\newblock {\em Nonlinearity}, 35:3417, 2022.

\bibitem{CHEN2019}
C.~Chen and X.~Yang.
\newblock Fast, provably unconditionally energy stable, and second-order accurate algorithms for the anisotropic {C}ahn–{H}illiard model.
\newblock {\em Comput. Methods Appl. Mech. Eng.}, 351:35--59, 2019.

\bibitem{CHEN2022}
R.~Chen and S.~Gu.
\newblock On novel linear schemes for the {C}ahn–{H}illiard equation based on an improved invariant energy quadratization approach.
\newblock {\em J. Comput. Appl. Math.}, 414:114405, 2022.

\bibitem{Diegel}
A.~E. Diegel, C.~Wang, and S.~M. Wise.
\newblock {Stability and convergence of a second-order mixed finite element method for the Cahn–Hilliard equation}.
\newblock {\em IMA J. Numer.}, 36:1867--1897, 2015.

\bibitem{Elliott1992}
C.~M. Elliott and S.~Larsson.
\newblock Error estimates with smooth and nonsmooth data for a finite element method for the {C}ahn-{H}illiard equation.
\newblock {\em Math. Comput.}, 58(198):603, May 1992.

\bibitem{Feng2004}
X.~Feng and A.~Prohl.
\newblock Error analysis of a mixed finite element method for the {C}ahn-{H}illiard equation.
\newblock {\em Numer. Math.}, 99(1):47–84, Sept. 2004.

\bibitem{GONG2019}
Y.~Gong and J.~Zhao.
\newblock Energy-stable {R}unge-{K}utta schemes for gradient flow models using the energy quadratization approach.
\newblock {\em Appl. Math. Lett.}, 94:224--231, 2019.

\bibitem{Jngel2016}
A.~J\"{u}ngel.
\newblock {\em Entropy Methods for Diffusive Partial Differential Equations}.
\newblock Springer, Cham, 2016.

\bibitem{Juengel2021}
A.~J\"{u}ngel and A.~Zurek.
\newblock A convergent structure-preserving finite-volume scheme for the {S}higesada--{K}awasaki--{T}eramoto population system.
\newblock {\em SIAM J. Numer. Anal.}, 59:2286--2309, 2021.

\bibitem{JngelVetter+2023}
A.~Jüngel and M.~Vetter.
\newblock A convergent entropy-dissipating {BDF2} finite-volume scheme for a population cross-diffusion system.
\newblock {\em Comput. Methods Appl. Math.}, 2023.

\bibitem{Juengel2022}
A.~Jüngel and A.~Zurek.
\newblock {A discrete boundedness-by-entropy method for finite-volume approximations of cross-diffusion systems}.
\newblock {\em IMA J. Numer. Anal.}, 43:560--589, 2022.

\bibitem{LI23}
Y.~Li and J.~Yang.
\newblock Consistency-enhanced {SAV} {BDF2} time-marching method with relaxation for the incompressible {C}ahn–{H}illiard–{N}avier–{S}tokes binary fluid model.
\newblock {\em Commun. Nonlinear Sci. Numer.}, 118:107055, 2023.

\bibitem{Liu2016}
Y.~Liu, W.~Chen, C.~Wang, and S.~M. Wise.
\newblock Error analysis of a mixed finite element method for a {C}ahn-{H}illiard-{H}ele-{S}haw system.
\newblock {\em Numer. Math.}, 135:679--709, 2016.

\bibitem{Paul}
M.~Luk\'{a}\voo{c}ov\'{a}-Medvid'ov\'{a}, P.~J. Strasser, B.~D\"{u}nweg, , and N.~Tretyakov.
\newblock Energy-stable numerical schemes for multiscale simulations of polymer-solvent mixtures.
\newblock In {\em Mathematical Analysis of Continuum Mechanics and Industrial Applications II (eds. van Meurs, Kimura, Notsu}, pages 153--165. Springer, Singapore, 2018.

\bibitem{Schmid2023}
F.~Schmid.
\newblock Understanding and modeling polymers: The challenge of multiple scales.
\newblock {\em ACS Polym. Au}, 3:28--58, 2023.

\bibitem{SHEN2018407}
J.~Shen, J.~Xu, and J.~Yang.
\newblock The scalar auxiliary variable ({SAV}) approach for gradient flows.
\newblock {\em J. Comput. Phys.}, 353:407--416, 2018.

\bibitem{Strasser.2019}
P.~J. Strasser, G.~Tierra, B.~D{\"u}nweg, and M.~Luk\'{a}\voo{c}ov\'{a}-Medvid'ov\'{a}.
\newblock {Energy-stable linear schemes for polymer--solvent phase field models}.
\newblock {\em {Comput. Math. Appl.}}, 77:125--143, 2019.

\bibitem{Tanaka}
H.~Tanaka.
\newblock Viscoelastic phase separation.
\newblock {\em J. Condens. Matter Phys.}, 12:R207--R264, 2000.

\bibitem{Tanaka2017}
H.~Tanaka.
\newblock Phase separation in soft matter: the concept of dynamic asymmetry.
\newblock In L.~Bocquet, D.~Quere, T.~A. Witten, and L.~F. Cugliandolo, editors, {\em Soft {Interfaces}: {Lecture} {Notes} of the {Les} {Houches} {Summer} {School}: {Volume} 98, {July} 2012}. Oxford University Press, Oxford, 2017.

\bibitem{Wloka}
J.~Wloka.
\newblock {\em Partial differential equations}.
\newblock Cambridge University Press, Cambridge, 1987.
\newblock Translated from the German by C. B. Thomas and M. J. Thomas.

\bibitem{Yang2020}
X.~Yang and G.-D. Zhang.
\newblock Convergence analysis for the invariant energy quadratization ({IEQ}) schemes for solving the {C}ahn–{H}illiard and {A}llen–{C}ahn equations with general nonlinear potential.
\newblock {\em J. Sci. Comput.}, 82, 2020.

\bibitem{Zeidler1}
E.~Zeidler.
\newblock {\em Nonlinear Functional Analysis and its Applications~{I}: Fixed-Point Theorems}.
\newblock Springer, New York, 1986.

\bibitem{Zhang2022}
Z.~Zhang, Y.~Gong, and J.~Zhao.
\newblock A remark on the invariant energy quadratization ({IEQ}) method for preserving the original energy dissipation laws.
\newblock {\em Electron. Res. Arch.}, 30:701--714, 2022.

\bibitem{ZZE}
D.~Zhou, P.~Zhang, and W.~E.
\newblock {Modified models of polymer phase separation}.
\newblock {\em {Phys. Rev. E}}, 73:061801, 2006.

\end{thebibliography}

\end{document}